\documentclass[12pt]{article}
\usepackage{amsfonts}
\usepackage{amsthm}
\input{epsf.sty}
\usepackage{latexsym,amsmath,amsfonts,amscd}
\usepackage{epsfig}
\usepackage{changebar}
\usepackage{pstricks}
\usepackage{pst-plot}
\usepackage{multirow}
\usepackage{subfigure}
\usepackage{placeins}
\usepackage{amssymb}
\usepackage{mathrsfs}
\usepackage[title]{appendix}
\usepackage{bm}
\usepackage{hyperref}

\newtheorem{proposition}{Proposition}[section]
\newtheorem{rem}{Remark}[section]
\newtheorem{examp}{Example}[section]

\newcommand{\brac}[1]{\left(#1\right)}
\newcommand{\abs}[1]{\left\vert#1\right\vert}
\newcommand{\norm}[1]{\left\Vert#1\right\Vert}
\def\half{\frac 1 2}

\newcommand{\bk}{\mathbf{k}}
\newcommand{\bx}{\mathbf{x}}

\newfont{\iams}{msbm9}

\newcommand{\commentbis}[1]{}
\newcommand{\be}{\begin{eqnarray}}
\newcommand{\ee}{\end{eqnarray}}
\newcommand{\beno}{\begin{eqnarray*}}
\newcommand{\eeno}{\end{eqnarray*}}

\newcommand{\barr}[1]{\begin{array}{#1}}
\newcommand{\earr}{\end{array}}

\newcommand{\beq}{\begin{equation}}
\newcommand{\eeq}{\end{equation}}
\newcommand{\beqa}{\begin{eqnarray}}
\newcommand{\eeqa}{\end{eqnarray}}

\newcommand{\bv}{{\bf v}}

\newcommand{\bV}{{\bf V}}
\newcommand{\bn}{{\bf n}}

\newcommand{\Rd}{{\mathbb{R}^d}}

\newcommand{\bP}{{\bf P}}

\newcommand{\bzero}{\mathbf{0}}
\newcommand{\bone}{\mathbf{1}}
\newcommand{\bl}{\mathbf{l}}
\newcommand{\bi}{\mathbf{i}}

\newcommand{\bj}{\mathbf{j}}
\newcommand{\bW}{\mathbf{W}}

\newcommand{\bal}{{\bm{\alpha}}}
\newcommand{\bb}{{\bm{\beta}}}

\allowdisplaybreaks

\title
{An adaptive multiresolution discontinuous Galerkin method with artificial viscosity for scalar hyperbolic conservation laws in multidimensions}

\author{
Juntao Huang
\thanks{Department of Mathematics, Michigan State University, East Lansing, MI 48824, USA. E-mail: huangj75@msu.edu.}
\and
 Yingda Cheng
\thanks{Department of Mathematics, Department of  Computational Mathematics, Science and Engineering, Michigan State University,
               East Lansing, MI 48824, USA.
E-mail: ycheng@msu.edu. Research is supported by NSF grants  DMS-1453661 and DMS-1720023.}
}

\begin{document}

%


\maketitle

\begin{abstract}
In this paper, we develop an adaptive multiresolution discontinuous Galerkin (DG) scheme for scalar hyperbolic conservation laws in multidimensions. Compared with previous work  for linear hyperbolic equations \cite{guo2016transport, guo2017adaptive}, a class of   interpolatory multiwavelets are   applied to efficiently compute the nonlinear integrals over elements and edges in DG schemes. The resulting algorithm, therefore can achieve similar computational complexity as the sparse grid DG method for smooth solutions. Theoretical and numerical studies are performed taking into consideration of accuracy and stability with regard to the choice of the interpolatory multiwavelets.  Artificial viscosity is added to capture the shock and only acts on the leaf elements taking advantages of the multiresolution representation. Adaptivity is realized by auto error thresholding based on   hierarchical surplus. Accuracy and robustness are demonstrated by several numerical tests.
\end{abstract}

\bigskip

\vfill


\noindent {\bf Keywords:} discontinuous Galerkin methods; multiresolution analysis; sparse grids; hyperbolic conservation laws; artificial viscosity.


\section{Introduction}
\setcounter{equation}{0}
\setcounter{figure}{0}
\setcounter{table}{0}

In this paper, we develop an adaptive multiresolution discontinuous Galerkin (DG) method for scalar nonlinear conservation laws in multidimensional case:
\begin{equation}\label{eq:cons-law}
    \partial_t u + \nabla\cdot {\bf f}(u) = 0, \quad (\bx,t) \in \Omega\times(0,T],
\end{equation}
with appropriate initial and   boundary conditions. Here $\Omega\subset\Rd$, $u=u(\bx,t)$ is the unknown function, and ${\bf f}(u)=(f_1(u),f_2(u),\dots,f_d(u))$ is the physical flux. We assume $\Omega=[0,1]^d$ in the paper, but the discussion can be easily generalized to arbitrary box-shaped domains.  

The DG method is a class of finite element methods using discontinuous approximation space for the numerical solutions and the test functions. The Runge-Kutta DG scheme   for hyperbolic equations became very popular due to its provable stability and convergence, excellent conservation properties and accommodation for adaptivity and parallel implementations. We refer readers to the review papers \cite{cockburn2001review,cockburn2000development} for details. To adapt the degrees of freedom according to the local behavior of the numerical solution, many kinds of \emph{a posteriori} error estimates have been designed for the DG schemes for hyperbolic equations, see e.g. \cite{bey1996hp,adjerid2002posteriori,remacle2003adaptive,hartmann2002adaptive,houston2002hp,hartmann2003adaptive}.  On the other hand, by using multiresolution analysis (MRA), automatic adaptivity can be achieved and no additional  \emph{a posteriori}  error indicator is needed. Such ideas have been used to  accelerate the computations for conservation laws under finite difference or finite volume frameworks \cite{harten_multiresolution_1995,bihari_multiresolution_1997,dahmen_multiresolution_2001,alves_adaptive_2002,cohen_fully_2003,chiavassa2005multiresolution} and were used as trouble cell indicators for DG methods \cite{vuik2014multiwavelet}. In recent years, there have been interests in developing  adaptive multiresolution DG schemes  \cite{calle_wavelets_2005}. In particular,   multiresolution-based adaptive DG schemes for solving one dimensional scalar conservation laws were proposed by M{\"u}ller \textit{et al.} in \cite{hovhannisyan2014scalar} and further extended to multidimensional cases \cite{gerhard2016mutidim,gerhard2017thesis}, compressible flows \cite{iacono2011compressible,gerhard2015compressible} and shallow water equations \cite{gerhard2015shallow,kesserwani2015shallow}.  {  The key idea is to perform a multiresolution analysis  using multiwavelets on a hierarchy of nested grids for the data given on a uniformly refined mesh. With such an approach, the scheme on a uniformly refined mesh is computed   on a locally refined adapted subgrid while preserving the accuracy. }

Another idea to utilize the computational advantages of  the MRA framework is called the sparse grid method \cite{bungartz2004sparse}, which is a well-known tool to compute high-dimensional PDEs and stochastic differential equations.   
Based on the attractive features of DG methods for solving convection-dominated problems, in recent years, we initiated a line of research developing the  (adaptive) sparse grid DG methods, including the work for elliptic equations \cite{wang2016elliptic}, transport equations \cite{guo2016transport}, reaction-diffusion equations \cite{liu2019krylov} and Vlasov-Maxwell equations \cite{tao2018vlasov}.  For smooth solutions, the schemes we constructed  can successfully reduce the number of degrees of freedom (DoF) of unknown from $\mathcal{O}(h^{-d})$ to $\mathcal{O}(h^{-1}|\log_2h|^{d-1})$ for $d$-dimensional problem, where $h$ is the uniform mesh size in each dimension. Stability and conservation of standard DG methods can be maintained. Errors are only slightly deteriorated for smooth solutions. Adaptivity can be incorporated naturally to treat solutions with less smoothness or local structures.  {  This is the line of research we continue in this paper. }

However, the main bottleneck for the sparse grid DG scheme developed so far is that, it is mainly for \emph{``linear"} equations. Here,  \emph{``linear"} refers to either linear variable coefficient equations with given coefficients or  coefficients that have some specified dependence on the unknowns, e.g. Vlasov systems through self-consistent field. There  remain significant challenges to extend the methods to truly nonlinear problems in an efficient manner. For example,   previous work in the literature on adaptive multiresolution DG schemes resort to the finest scale for the actual time evolution for the nonlinear terms. Therefore, the computational cost is proportional to the number of cells on the finest level,  i.e. $\mathcal{O}(h^{-d})$ operations, and the reduced DoF in the solution representation is not realized in the actual computation.    For nonlinear equations, there is only limited literature on collocation or finite difference based sparse grid methods \cite{griebel1999adaptive}, and the order of accuracy of the schemes is low.  Sparse grid combination methods work for nonlinear problems, but they are less flexible in terms of adaptivity \cite{lastdrager2001sparse}.

{  This work serves as a proof-of-concept for   a systematic approach for  adaptive sparse grid DG method to solve  nonlinear PDEs in high dimensions. We  construct a scheme that can recover the computational complexity of the sparse grid method for smooth solutions.  This is different from previous approaches for adaptive multiresolution DG methods \cite{gerhard2013adaptive,gerhard2016mutidim,archibald2011adaptive,shelton_phdthesis}.
The evolution in our scheme is carried out by the multiresolution basis functions, and it can recover the computational efficiency of sparse grid approaches for smooth solutions in high dimensions. In particular, our methods never convert the multiscale coefficients to single scale coefficients.  This also raises many computational challenges as mentioned before.
}
{   To compute  nonlinear terms, we use sparse grid collocation methods introduced in \cite{tao2019collocation}, where the idea is to design new multiwavelets associated with interpolation on nested grids and gives a framework to design adaptive sparse grid collocation onto arbitrary high order piecewise polynomial spaces. Moreover, the algorithm converting between the point values and the derivatives to the coefficients of the hierarchical wavelets can be performed efficiently by  the fast wavelet transform \cite{tao2019collocation}.
In this paper, we approximate the nonlinear integral terms in the semi-discrete DG scheme  by integrating a linear combination of collocation wavelets up to desired order of accuracy. We analyze the truncation error of the DG scheme with the interpolation following the approach in \cite{DG4,huang2017quadrature}. 
It is shown that, for the interpolation we need polynomials of one degree higher than in the original DG function space, if one would like to preserve the order of accuracy for the original (standard or sparse grid) DG scheme.  }  Then, we compare different ways of sparse grid collocation methods and find that there exists some {  \emph{ instability}} for Lagrange interpolation when solving \eqref{eq:cons-law} especially for DG space with higher degrees of polynomial (see Table 2 and Table 3 in Section 4). This motivates us to apply Hermite interpolation, which is more stable than the Lagrange interpolation \cite{hagstrom2015hermite}.   {  Because the Alpert's multiwavelets and the interpolation multiwavelet bases are both global, the evaluation of the residual yields denser matrix than those obtained by standard local bases. We use fast matrix-vector multiplication  e.g. those developed for sparse grid methods \cite{shen2010efficient,zeiser2011fast}  to recover efficient computational scaling for such calculations.}


Another challenge we address in the paper is how to capture the shock and   entropy solutions to \eqref{eq:cons-law}. There are two approaches  in the literature. The first one is to apply limiters to control spurious oscillations and at the same time maintain accuracy in smooth regions, e.g., the minmod-type limiter \cite{DG2}, the moment-based limiter \cite{biswas1994parallel} and WENO limiter \cite{qiu2005runge}. However, it is quite difficult to impose   limiters in the sparse grid DG methods, due to the global feature of the basis functions. Also a preliminary calculation from us shows that the piecewise constant sparse grid DG method in multidimensions is not monotone. This motivates us to use the second approach, which is to add artificial viscosity, see e.g. \cite{bassi1995viscosity,persson2006shock,guermond2011viscosity,klockner2011viscous,kornelus2017scaling}. The idea is to add a diffusion term in the equation where the diffusion coefficient  vanishes in the smooth region and becomes non-zero near the shock. This can be achieved by  techniques such as entropy production \cite{guermond2011viscosity} or local smoothness indicator \cite{persson2006shock}.
We add an artificial viscosity term following the approach in \cite{bassi1995viscosity}. Based on the estimate of the magnitudes of coefficients of hierarchical basis functions in \cite{guo2016transport,guo2017adaptive}, we propose a smoothness indicator, which is built upon the inherent MRA and can automatically pick out the discontinuous regions. To improve the computational efficiency of our scheme, the implicit-explicit (IMEX) Runge-Kutta time integration is applied, where the nonlinear convection term is treated explicitly and the linear diffusion term is computed implicitly.

The rest of this paper is organized as follows. In Section 2, we review MRA associated with two sets of basis functions, i.e., the Alpert's multiwavelets \cite{alpert1993} and the interpolatory multiwavelets \cite{tao2019collocation}. The adaptive multiresolution DG scheme is constructed in Section 3 using both sets of multiwavelets. The numerical performance is validated  by linear advection equations, Burgers' equations and  KPP problems in Section 4. We conclude the paper in Section 5. The appendix collects the explicit formulas of the interpolatory multiwavelets used in this paper.

\section{MRA and multiwavelets}

In this section, we review MRA associated with piecewise polynomial space. 
We will   start with  Alpert's multiwavelets \cite{alpert1993} and then review the interpolatory multiwavelets \cite{tao2019collocation}.

\subsection{Alpert's multiwavelets}
\label{sec:alpert}

We first review MRA achieved by Alpert's basis functions in one dimension \cite{alpert1993}. We define a set of nested grids, where the $n$-th level grid $\Omega_n$ consists of $2^n$ uniform cells
\begin{equation*}
  I_{n}^j=(2^{-n}j, 2^{-n}(j+1)], \quad j=0, \ldots, 2^n-1
\end{equation*}
for $n \ge 0.$ For notational convenience, we also denote $I_{-1}=[0,1].$
The usual piecewise polynomial space of degree at most $k\ge1$ on the $n$-th level grid $\Omega_n$ for $n\ge 0$ is denoted by
\begin{equation}\label{eq:DG-space-Vn}
V_n^k:=\{v: v \in P^k(I_{n}^j),\, \forall \,j=0, \ldots, 2^n-1\}.
\end{equation}
Then, we have the nested structure
$$V_0^k \subset V_1^k \subset V_2^k \subset V_3^k \subset  \cdots$$
We can now define the multiwavelet subspace $W_n^k$, $n=1, 2, \ldots $ as the orthogonal complement of $V_{n-1}^k$ in $V_{n}^k$ with respect to the $L^2$ inner product on $[0,1]$, i.e.,
\begin{equation*}
V_{n-1}^k \oplus W_n^k=V_{n}^k, \quad W_n^k \perp V_{n-1}^k.
\end{equation*}
For notational convenience, we let
$W_0^k:=V_0^k$, which is the standard polynomial space of degree up to $k$ on $[0,1]$. Therefore, we have $V_n^k=\bigoplus_{0 \leq l \leq n} W_l^k$.

Now we  define a set of orthonormal basis associated with the space $W_l^k$. The case of mesh level $l=0$ is trivial. We use the normalized shifted Legendre polynomials in $[0,1]$ and denote the basis by $v^0_{i,0}(x)$ for $i=0,\ldots,k$.
When $l>0$, the orthonormal bases in $W_l^k$ are presented in \cite{alpert1993} and denoted by 
$$v^j_{i,l}(x),\quad i=0,\ldots,k,\quad j=0,\ldots,2^{l-1}-1.$$
The construction follows a repeated Gram-Schmidt process and the explicit expression of the multiwavelet basis functions are provided in Table 1 in \cite{alpert1993}.
Note that such multiwavelet bases retain the orthonormal property of wavelet bases for different mesh levels, i.e.,
\begin{equation}\label{ortho1d}
\int_0^1 v^j_{i,l}(x)v^{j'}_{i',l'}(x)\,dx=\delta_{ii'}\delta_{ii'}\delta_{jj'}.
\end{equation}
and the support of $v^j_{i,l}$ is in $I_{l-1}^j$.

Multidimensional case when $d>1$ follows from a tensor-product approach. First we recall some basic notations. For a multi-index $\mathbf{\alpha}=(\alpha_1,\cdots,\alpha_d)\in\mathbb{N}_0^d$, where $\mathbb{N}_0$  denotes the set of nonnegative integers, the $l^1$ and $l^\infty$ norms are defined as 
$$
|\bal|_1:=\sum_{m=1}^d \alpha_m, \qquad   |\bal|_\infty:=\max_{1\leq m \leq d} \alpha_m.
$$
The component-wise arithmetic operations and relational operations are defined as
$$
\bal \cdot \bb :=(\alpha_1 \beta_1, \ldots, \alpha_d \beta_d), \qquad c \cdot \bal:=(c \alpha_1, \ldots, c \alpha_d), \qquad 2^\bal:=(2^{\alpha_1}, \ldots, 2^{\alpha_d}),
$$
$$
\bal \leq \bb \Leftrightarrow \alpha_m \leq \beta_m, \, \forall m,\quad
\bal<\bb \Leftrightarrow \bal \leq \bb \textrm{  and  } \bal \neq \bb.
$$

By making use of the multi-index notation, we denote by $\bl=(l_1,\cdots,l_d)\in\mathbb{N}_0^d$ the mesh level in a multivariate sense. We  define the tensor-product mesh grid $\Omega_\bl=\Omega_{l_1}\otimes\cdots\otimes\Omega_{l_d}$ and the corresponding mesh size $h_\bl=(h_{l_1},\cdots,h_{l_d}).$ Based on the grid $\Omega_\bl$, we denote  $I_\bl^\bj=\{\bx:x_m\in(h_mj_m,h_m(j_{m}+1)),m=1,\cdots,d\}$ as an elementary cell, and 
$$\bV_\bl^k:=\{\bv: {  \bv \in Q^k(I^{\bj}_{\bl})}, \,\,  \bzero \leq \bj  \leq 2^{\bl}-\bone \}= V_{l_1,x_1}^k\times\cdots\times  V_{l_d,x_d}^k$$
as the tensor-product piecewise polynomial space, where $Q^k(I^{\bj}_{\bl})$ represents the collection of polynomials of degree up to $k$ in each dimension on cell $I^{\bj}_{\bl}$. 
If we use equal mesh refinement of size $h_N=2^{-N}$ in each coordinate direction, the  grid and space will be denoted by $\Omega_N$ and $\bV_N^k$, respectively.  

Based on a tensor-product construction, the multidimensional increment space can be  defined as
$$\bW_\bl^k=W_{l_1,x_1}^k\times\cdots\times  W_{l_d,x_d}^k.$$
Therefore,  the standard tensor-product piecewise polynomial space on $\Omega_N$ can be written as
\begin{equation}\label{eq:hiere_tp}
\bV_N^k=\bigoplus_{\substack{ |\bl|_\infty \leq N\\\bl \in \mathbb{N}_0^d}} \bW_\bl^k,
\end{equation}
while the sparse grid approximation space in \cite{wang2016elliptic} is
\begin{equation}
\label{eq:hiere_sg}
\hat{\bV}_N^k:=\bigoplus_{\substack{ |\bl|_1 \leq N\\\bl \in \mathbb{N}_0^d}}\bW_\bl^k \subset \bV_N^k.
\end{equation}
The dimension of $\hat{\bV}_N^k$ scales as $O((k+1)^d2^NN^{d-1})$ \cite{wang2016elliptic}, which is significantly less than that of $\bV_N^k$ with exponential dependence on $Nd$. The approximation results for $\hat{\bV}_N^k$ are discussed in \cite{wang2016elliptic,guo2016transport}, which has a stronger smoothness requirement than the traditional $\bV_N^k$ space. In this paper, we will not require the numerical solution to be in $\hat{\bV}_N^k$, but rather in $\bV_N^k$ and to be chosen adaptively similar to \cite{guo2017adaptive}.

The basis functions in multidimensions are defined as
\begin{equation}\label{eq:multidim-basis}
v^\bj_{\bi,\bl}(\bx) := \prod_{m=1}^d v^{j_m}_{i_m,l_m}(x_m),
\end{equation}
for $\bl \in \mathbb{N}_0^d$, $\bj \in B_\bl := \{\bj\in\mathbb{N}_0^d: \,\mathbf{0}\leq\bj\leq\max(2^{\bl-\mathbf{1}}-\mathbf{1},\mathbf{0}) \}$ and $\mathbf{1}\leq\bi\leq \bk+\mathbf{1}$. The orthonormality of the bases can be established by \eqref{ortho1d}.

{ 
There exists the estimate of the coefficients of hierarchical basis functions in \cite{guo2016transport,guo2017adaptive}, which is also a classical result from wavelet theory and holds in a more general context \cite{cohen2003}. It is shown in \cite{guo2016transport,guo2017adaptive} that for a function {  $u \in H^{p+1}(\Omega)$},
\begin{equation}\label{eq:estimate-coeff-sum-j}
     (\sum_{\substack{ \mathbf{0}\leq \mathbf{j} \leq \max(2^{\mathbf{l}-1}-\mathbf{1},\mathbf{0})\\ \mathbf{1}\leq\bi\leq \bk+\mathbf{1}}}|u^\bj_{\bi,\bl}|^2  )^{1/2} \leq C2^{-(q+1)|\bl|_1}|u|_{H^{p+1}(\Omega)},    
\end{equation}
where    $q=\min\{p,k\},$ and $C$ is a constant independent of mesh level $\bl$. Therefore, by assuming that $u\in W^{p+1,\infty}(I_{\bl}^{\bj})$, we can obtain a local estimate   in each element on mesh level $\bl$: for any index $\bj$,
\begin{equation}
     (\sum_{\mathbf{1}\leq\bi\leq \bk+\mathbf{1}} |u^\bj_{\bi,\bl}|^2  )^{1/2} \leq C2^{-(q+\half)|\bl|_1} |u|_{W^{p+1,\infty}(\Omega)}.
\end{equation}
Therefore, for sufficiently smooth functions, the coefficients should decay like
\begin{equation}\label{eq:smooth-coeff-l2-norm}
     (\sum_{\mathbf{1}\leq\bi\leq \bk+\mathbf{1}} |u^\bj_{\bi,\bl}|^2  )^{1/2} \sim 2^{-(k+\half)|\bl|_1}.
\end{equation}
This will be used in the construction of the smoothness indicator in Section 3.
}

\subsection{Interpolatory multiwavelets}\label{subsec:interp-basis}
Alpert's multiwavelets and the space $W_l^k$ are constructed so that they correspond to the difference of the $L^2$ projection on adjacent levels. The idea of the sparse grid collocation basis proposed in \cite{tao2019collocation} is to switch the operator to be interpolation on nested grids. Below, we will outline the construction. Denote the set of interpolation points in the interval $I=[0,1]$ at mesh level 0 by $X_0 = \{ x_i \}_{i=0}^P\subset I$. Here, the number of points in $X_0$ is $(P+1)$. Then the interpolation points at mesh level $n\ge1$, $X_n$ can be obtained correspondingly as
\begin{equation*}
    X_n = \{ x_{i,n}^j := 2^{-n}(x_i+j), \quad i=0,\dots,P, \quad j=0,\dots,2^{n}-1 \}.
\end{equation*}
We require the points to be nested, i.e. 
\begin{equation}
\label{nestpts}
    X_0 \subset X_1 \subset X_2 \subset X_3 \subset \cdots.
\end{equation}
to save computational cost. 
This can be achieved by requiring  $X_0\subset X_1$, and then one can deduce \eqref{nestpts} easily.

Given the nodes, we define the basis functions on the $0$-th level grid as  Lagrange $(K=0)$ or Hermite $(K\ge1)$ interpolation polynomials of degree $\le M:=(P+1)(K+1)-1$ which satisfy the property:
\begin{equation*}
    \phi_{i,l}^{(l')}(x_{i'}) = \delta_{ii'}\delta_{ll'},
\end{equation*}
for $ i,i'=0,\dots,P$ and $l,l'=0,\dots,K$. It is easy to see that
$
\textrm{span} \{ \phi_{i,l}, \quad i=0,\dots,P, \quad l=0,\dots,K \}=V_0^M.  
$
The constants $P, K, M$ will be specified later on in the paper.
With the basis function at mesh level 0, we can define basis function at mesh level $n\ge1$:
\begin{equation*}
  \phi_{i,l,n}^j:=2^{-nl}\phi_{i,l}(2^nx-j), \quad i=0,\dots,P, \quad l=0,\dots,K, \quad j=0,\dots,2^n-1 
\end{equation*}
which is a complete basis set for $V_n^M.$


Next, we introduce the hierarchical representations. Define $\tilde{X}_0 := X_0$ and {  $\tilde{X}_n := X_n\backslash X_{n-1}$} for $n\ge1$, then we have the decomposition
\begin{equation*}
    X_n = \tilde{X}_0 \cup \tilde{X}_1 \cup \cdots \cup \tilde{X}_n .
\end{equation*}
Denote the points in $\tilde{X}_1$ by $\tilde{X}_1=\{ \tilde{x}_i \}_{i=0}^P$. Then the points in $\tilde{X}_n$ for $n\ge1$ can be represented by
\begin{equation*}
    \tilde{X}_n = \{ \tilde{x}_{i,n}^j:=2^{-(n-1)}(\tilde{x}_i+j), \quad  i=0,\dots,P, \quad j=0,\dots,2^{n-1}-1 \}.
\end{equation*}


For notational convenience, we let $\tilde{W}_0^M:=V_0^M.$ The increment function space $\tilde{W}_n^M$ for $n\ge1$ is introduced as a function space    that satisfies
\begin{equation}\label{eq:func-space-sum}
    V_n^M = V_{n-1}^M \oplus \tilde{W}_n^M,
\end{equation}
and is defined through the multiwavelets
 $\psi_{i,l} \in V_1^M$ that satisfies
\begin{equation*}
    \psi_{i,l}^{(l')}(x_{i'}) = 0, \quad \psi_{i,l}^{(l')}(\tilde{x}_{i'}) = \delta_{i,i'}\delta_{l,l'},
\end{equation*}
for $i,i'=0,\dots,P$ and $l,l'=0,\dots,K$. Here the superscript $(l')$ denotes the $l'$-th order derivative. Then  $W_n^M$ is given by
\begin{equation*}
    \tilde{W}_n^M = \textrm{span} \{ \psi_{i,l,n}^j, \quad i = 0,\dots,P, \quad l=0,\dots,K, \quad j=0,\dots,2^{n-1}-1 \}
\end{equation*}
where $\psi_{i,l,n}^j(x) := 2^{-(n-1)l}\psi_{i,l}(2^{n-1}x-j)$.  For completeness, we list the basis functions used in this paper in the appendix. 

The construction above has close connection with interpolation operators.   For a given function $f(x)\in C^{K+1}(I)$, we define $\mathcal{I}^{P,K}_{N}[f]$ as the standard Hermite interpolation on $V^{M}_{N},$ and have the representation
\begin{align*}
	\mathcal{I}^{P,K}_{N}[f](x) 
	= \sum_{n=0}^{N} \sum _{j=0}^{\max(2^{n-1}-1,0)}\sum_{l=0}^K \sum_{i=0}^{P} b_{i,l,n}^{j} \psi^{j}_{i,l,n}(x).
\end{align*}
Clearly, $(\mathcal{I}^{P,K}_{n}-\mathcal{I}^{P,K}_{n-1})[f](x) \in  \tilde{W}_n^M.$ 
The algorithm converting between the point values and the derivatives $\{ f^{(l)}(x_{i,n}^j) \}$ to hierarchical coefficients $\{ b_{i,l,n}^{j} \}$ is given in \cite{tao2019collocation}, and by a standard  argument in fast wavelet transform, can be performed in $O(M2^n)$ flops.

The multidimensional construction follows similar lines as in Section \ref{sec:alpert}. We let 
$$\tilde{\bW}_\bl^M=\tilde{W}_{l_1,x_1}^M\times\cdots\times  \tilde{W}_{l_d,x_d}^M,$$
Therefore,   
\begin{equation*}
\bV_N^M=\bigoplus_{\substack{ |\bl|_\infty \leq N\\\bl \in \mathbb{N}_0^d}} \tilde{\bW}_\bl^M,
\end{equation*}
while the sparse grid approximation space is
\begin{equation*}
\hat{\bV}_N^M=\bigoplus_{\substack{ |\bl|_1 \leq N\\\bl \in \mathbb{N}_0^d}}\tilde{\bW}_\bl^M.
\end{equation*}
Note that the construction by Alpert's multiwavelet and the interpolatory multiwavelet gives the same sparse grid space. 
Finally, the interpolation operator in multidimension $\mathcal{I}^{P,K}_{N}: C^{K+1}(\Omega)\rightarrow \mathbf{V}^{M}_{N}$:
\begin{align*}
	\mathcal{I}^{P,K}_{N}[f](\mathbf{x}) 
	= \sum_{ \substack{ \abs{\mathbf{n}}_{\infty}\leq N \\ \mathbf{0}\leq \mathbf{j} \leq \max(2^{\mathbf{n}-1}-\mathbf{1},\mathbf{0}) \\ \mathbf{0}\leq \mathbf{l}\leq \mathbf{K} \\ \mathbf{0}\leq \mathbf{i}\leq \mathbf{P}  } } b^{\mathbf{j}}_{\mathbf{i}, \mathbf{l}, \mathbf{n}} \psi^{\mathbf{j}}_{\mathbf{i}, \mathbf{l}, \mathbf{n}} (\mathbf{x}),
\end{align*}
where the multidimensional basis functions $\psi^{\mathbf{j}}_{\mathbf{i}, \mathbf{l}, \mathbf{n}} (\mathbf{x})$ are defined in the same approach as \eqref{eq:multidim-basis} by tensor products:
\begin{equation}\label{eq:multidim-basis-interpolation}
	\psi^{\mathbf{j}}_{\mathbf{i}, \mathbf{l}, \mathbf{n}} (\mathbf{x}) := \prod_{m=1}^d \psi^{j_m}_{i_m,l_m,n_m}(x_m),
\end{equation}
If the space is switched from $\mathbf{V}^{M}_{N}$ to some subset of $\mathbf{V}^{M}_{N},$ e.g. the sparse grid space $\hat{\bV}_N^M$ or some other subset of $\mathbf{V}^{M}_{N}$ that is dynamically chosen, the interpolation operator can be defined accordingly, taking only multiwavelet basis functions that belong to that space.

\section{Adaptive multiresolution DG evolution algorithm}

In this section, we will describe the adaptive multiresolution DG scheme for \eqref{eq:cons-law}. We will first introduce the DG scheme with multiresolution interpolation. The accuracy requirement for the interpolation operator is studied by local truncation error analysis. We then describe the adaptive strategy. Finally, the artificial viscosity is introduced based on the estimate of the coefficients of the hierarchical basis functions.

\subsection{DG scheme with multiresolution interpolation}

First, we review some basis notations about meshes.  Let $N$ be the maximum mesh level and $T_h$ be the collection of all elementary cell $I_N^{\bj}$, $0\le j_m\le 2^N-1$, $\forall m=1,\dots,d$. Define $\Gamma_h:=\bigcup_{K\in T_h} \partial K$ be the union of all the interfaces for all the elements in $T_h.$  Here, for simplicity,  we formulate the scheme with   periodic boundary conditions, while we keep in mind other boundary conditions can be treated in the DG framework as well.

The semi-discrete DG scheme for the scalar conservation law reads as \cite{DG2}
\begin{equation}\label{eq:DG-semi}
    \sum_{K\in T_h}\int_{K} (u_h)_t v_h dx - \sum_{K\in T_h}\int_{K}f(u_h)\cdot\nabla v_h dx + \sum_{e\in\Gamma_h}\int_{e}\widehat{f\cdot n_K}(u_h) v_h ds = 0.
\end{equation}
Here, $u_h$ is the numerical solution and $v_h$ is the test function. The numerical flux $\widehat{f\cdot n_K}(u_h)\equiv \widehat{f\cdot n_K}(u_h^{int}, u_h^{ext})$ is taken to be the global Lax-Friedrichs flux:
\begin{equation}\label{eq:global-LxF}
    \widehat{f\cdot n_K}(a, b) = \half \left(f(a) + f(b)\right)\cdot n_K - \half C(b-a),
\end{equation}
where $C=\max_{u}\abs{n_K\cdot f'(u)}$ and the maximum is taken over the whole domain. {  Note that the local Lax-Friedrichs flux can also be used with additional efforts in numerical interpolation.}
$u_h, v_h$ belong to the same function space $\bV^k.$  If $\bV^k=\bV_N^k,$ we recover the standard (or full grid) DG method.   If $\bV^k=\hat{\bV}_N^k,$ we obtain the sparse grid DG method. In this paper, we will take $\bV^k$ as   a subset of $\bV_N^k$ that is chosen adaptively as outlined in Section \ref{subsec:adapt}. 

In DG methods, the integrals over elements and edges are often approximated by numerical quadrature rules on each cell \cite{DG4}. However, in sparse grid DG method, this naive approach would result in computational cost that is proportional to the number of fundamental elements, i.e., $\mathcal{O}(h^{-d})$, and is still subject to the curse of dimensionality. 
To evaluate the integrals over elements and edges more efficiently with a cost proportional to the DoF of the underlying finite element space, we   interpolate the nonlinear function $f(u_h)$ by using the multiresolution Lagrange (or Hermite) interpolation basis functions introduced in Section \ref{subsec:interp-basis}. Therefore, the semi-discrete DG scheme with interpolation is
\begin{equation}\label{eq:DG-semi-interp}
    \sum_{K\in T_h}\int_{K} (u_h)_t v_h dx - \sum_{K\in T_h}\int_{K}\mathcal{I}[f(u_h)]\cdot\nabla v_h dx + \sum_{e\in\Gamma_h}\int_{e}\mathcal{I}[\widehat{f\cdot n_K}(u_h)]v_h ds = 0,
\end{equation}
where $\mathcal{I}[\cdot]$ is a multiresolution interpolation operator onto some finite element space  with the same multiresolution structure as $\bV^k,$ but of polynomial degree $M$. The choice of  $\mathcal{I}[\cdot]$ will be specified later, which plays important roles in numerical stability and accuracy. Note that the numerical flux $\widehat{f\cdot n_K}(u_h)$ is only defined at edges, thus it remains to clarify the meaning of the interpolation $\mathcal{I}[\widehat{f\cdot n_K}(u_h)]$. Since we use the global Lax-Friedrichs flux \eqref{eq:global-LxF}, we have
\begin{align*}
    \mathcal{I}[\widehat{f\cdot n_K}(u_h^{int}, u_h^{ext})] &= \half \left(\mathcal{I}[f(u_h^{int})] + \mathcal{I}[f(u_h^{ext})]\right)\cdot n_K - \half C(\mathcal{I}[u_h^{ext}]-\mathcal{I}[u_h^{int}]) \\
    &= \half \left(\mathcal{I}[f(u_h^{int})] + \mathcal{I}[f(u_h^{ext})]\right)\cdot n_K - \half C(u_h^{ext}-u_h^{int})
\end{align*}
due to the linearity of the interpolation operator $\mathcal{I}[\cdot]$. Therefore, we only need to obtain the interpolation $\mathcal{I}[f(u_h)]$ and then read the value on two sides of the edges to obtain $\mathcal{I}[f(u_h^{int})]$ and $\mathcal{I}[f(u_h^{ext})]$. Now, we discuss about numerical implementation. First, we  read the (derivative) values of $u_h$, which is a linear combination of Alpert's basis functions at the chosen interpolation points. Second, we calculate the (derivative) values of $f(u_h)$ at these interpolation points. Last, we transfer the (derivative) values to coefficients of interpolation basis, by using the algorithm introduced in \cite{tao2019collocation}. At this point, the numerical integrations can be performed through a fast matrix-vector product as in \cite{shen2010efficient}. We remark that the computational cost does not increase too much compared to the multiresolution DG schemes for linear equations introduced in \cite{guo2017adaptive}. The cost of the transformation from the (derivative) values to hierarchical coefficients is only \emph{linearly} dependent on the dimension $d$ \cite{tao2019collocation}.

Now we discuss the choice of $\mathcal{I}[\cdot].$ To preserve the accuracy of the original DG scheme \eqref{eq:DG-semi}, it is required that the interpolation operator $\mathcal{I}[\cdot]$ reaches certain accuracy. Following \cite{DG4}, we rewrite the weak formulation \eqref{eq:DG-semi-interp} in the ODE form as 
\begin{equation}
	\frac{du_h}{dt} = L_h(u_h),
\end{equation}
where $L_h(u)$ is an operator onto $\bV^k,$ which  is a discrete approximation of $-\nabla\cdot f(u)$ and satisfies
\begin{equation}\label{eq:def-Lh}
	\sum_{K\in T_h}\int_{K} L_h(u_h)v_hdx = \sum_{K\in T_h}\int_{K}\mathcal{I}[f(u_h)]\cdot\nabla v_h dx - \sum_{e\in\Gamma_h}\int_{e}\mathcal{I}[\widehat{f\cdot n_K}(u_h)] v_h ds.
\end{equation} To illustrate the ideas, we only consider the full grid or sparse grid DG methods, i.e. $\bV^k=\bV_N^k$ or $\bV^k=\hat{\bV}_N^k.$ For adaptive methods, similar intuitive arguments can be made, but rigorous proof is much harder.
Using   similar error estimates techniques in \cite{DG4,huang2017quadrature}, we have the following proposition on local truncation error:
\begin{proposition}[Accuracy of semi-discrete DG scheme with interpolation]\label{prop:accurate-interp}
	Assume that the DG finite element space (standard or sparse) has polynomials up to degree $k,$ if the interpolation operator in \eqref{eq:DG-semi-interp} has the accuracy of $h^{k+2}$ (standard) or $\abs{\log_2h}^{d}h^{k+2}$ (sparse) for sufficiently smooth functions, then the truncation error of the semi-discrete DG scheme with interpolation \eqref{eq:DG-semi-interp} is of order $h^{k+1}$ (standard) or $\abs{\log_2h}^{d}h^{k+1}$ (sparse). To be more precise, for sufficiently smooth function $u$, the standard DG with interpolation \eqref{eq:DG-semi-interp} has the truncation error:
	\begin{equation}\label{eq:truncation-regular}
		\norm{L_h(u)+\nabla\cdot f(u)}_{L^2(\Omega)}\le C h^{k+1},
	\end{equation}
	and the sparse grid DG with interpolation \eqref{eq:DG-semi-interp} has the truncation error:
	\begin{equation}\label{eq:truncation-sparse}
		\norm{L_h(u)+\nabla\cdot f(u)}_{L^2(\Omega)}\le C \abs{\log_2h}^{d}h^{k+1}.
	\end{equation}
    Here, the constant $C$ may depend on the solution, but does not depend on $h$.
\end{proposition}
\begin{proof}
	To save space, we only show the proof for full grid DG space $\bV^k=\bV_N^k$. Similar technique also applies to the sparse grid DG space using projection error estimates in \cite{guo2016transport}. 

    We denote the standard $L^2$ projection operator onto the standard DG finite element space by $\mathbb{P}$, then
	\begin{equation}
		\norm{L_h(u) + \nabla\cdot f(u)}_{L^{2}(\Omega)} \le e_1 + e_2,
	\end{equation}
	where 
	\begin{equation*}
		e_1 := \norm{L_h(u) + \mathbb{P}(\nabla\cdot f(u))}_{L^{2}(\Omega)},
	\end{equation*}
	and
	\begin{equation*}
		e_2 := \norm{\mathbb{P}(\nabla\cdot f(u)) - \nabla\cdot f(u)}_{L^{2}(\Omega)}.
	\end{equation*}
	The estimate for $e_2$ is trivial using projection properties:
	\begin{equation}\label{eq:estimate-e2}
		e_2 \le Ch^{k+1}\norm{\nabla\cdot f(u)}_{H^{k+1}(\Omega)}.
	\end{equation}	

	To estimate $e_1,$ we consider  any test function $v_h$ in DG space, and obtain  
	\begin{align*}
		&\sum_{K\in T_h}\int_{K}(L_h(u)+\mathbb{P}(\nabla\cdot f(u)))v_h =\sum_{K\in T_h}\int_{K}(L_h(u)+\nabla\cdot f(u))v_h \\
		={}& \sum_{K\in T_h}\int_K\mathcal{I}[f(u)]\cdot\nabla v_h - \sum_{e\in\Gamma_h}\int_{e} \mathcal{I}[\widehat{f\cdot n_K}(u)]\cdot n_K v_hds \\
		& - \sum_{K\in T_h}\int_K(f(u)\cdot\nabla v_h) + \sum_{e\in\Gamma_h}\int_{e} f(u)\cdot n_K v_hds \\
		={}& \sum_{{K\in T_h}}\int_K(\mathcal{I}[f(u)]-f(u))\cdot\nabla v_h - \sum_{e\in\Gamma_h}\int_{e} (\mathcal{I}[f(u)\cdot n_K]-{f}(u)\cdot n_K) v_hds \\
		\le{}& \norm{\mathcal{I}[f(u)]-f(u)}_{L^2{(\Omega)}}\norm{\nabla v_h}_{L^2(\Omega)} + \norm{\mathcal{I}[f(u)]-{f}(u)}_{L^2(\Gamma_h)}\norm{v_h}_{L^2(\Gamma_h)} \\
		\le{}& Ch^{k+2}h^{-1}\norm{v_h}_{L^2(\Omega)} + Ch^{-\frac12}h^{k+2}h^{-\frac12}\norm{v_h}_{L^2(\Omega)} \\
		={}& Ch^{k+1}\norm{v_h}_{L^2(\Omega)}.
	\end{align*}
    Here we use the multiplicative trace inequality and the inverse inequality, see e.g. Lemma 2.1 and Lemma 2.3 in \cite{huang2017quadrature}.
	We take $v_h$ to be $(L_h(u)+\mathbb{P}(\nabla\cdot f(u)))$ in the  inequality above and have
	\begin{equation*}
		\norm{L_h(u)+\mathbb{P}(\nabla\cdot f(u))}^2_{L^2(\Omega)} \le C h^{k+1} \norm{L_h(u)+\mathbb{P}(\nabla\cdot f(u))}_{L^2(\Omega)}
	\end{equation*}
	and eventually arrive at
	\begin{equation}\label{eq:estimate-e1}
		e_1 = \norm{L_h(u)+\mathbb{P}(\nabla\cdot f(u))}_{L^2(\Omega)}\le C h^{k+1}.
	\end{equation}	
    Combining \eqref{eq:estimate-e1} and \eqref{eq:estimate-e2}, we have the estimate for the truncation error \eqref{eq:truncation-regular}.
\end{proof}

\begin{rem}
 	From the proposition above, we find that, {  for the interpolation we need polynomials of one degree higher than in the original DG function space,} if one would like to preserve the order of accuracy for the original (standard or sparse grid) DG scheme, i.e. we shall require $M \ge k+1.$ For example, if we take  quadratic polynomials  for the DG space, then it is required to apply cubic interpolation operator (Lagrange or Hermite interpolation) to treat the nonlinear terms. {  From our numerical test, it seems that it is not a necessary condition for the standard DG method, but it is necessary for the sparse grid DG method.}
 \end{rem} 
    
\begin{rem}
	{  
	For the standard DG method, the collocation in our scheme for the volumes integrals should be equivalent to some quadrature formula depending on which interpolation operator is used. However, our method is not standard for the evaluation of the interfacial terms, which use point values (or derivatives) that may \emph{inside} the elements and not just on the interface.
	} 	
\end{rem} 
 
    In Proposition \ref{prop:accurate-interp}, we only estimate the truncation error, and this is far from {  a rigorous error estimate} that takes into account  stability. In   numerical experiments, we observe that the standard DG is stable with the Lagrange interpolation. However, the sparse grid DG with Lagrange interpolation is {  \emph{ unstable}} and will blow up with very fine mesh for polynomials of high degrees (see the numerical results in Table \ref{ex:table-2D-Burgers-Sparse-Lagr-inner} and Table \ref{ex:table-2D-Burgers-Sparse-Lagr-interface} in Section 4). With Hermite interpolation, the sparse grid DG scheme is more stable and produce satisfactory convergence rate (see Table \ref{ex:table-2D-Burgers-Sparse-Herm} in Section 4). Actually, for standard DG with quadrature rules applied in each element, if the truncation error satisfies the required order of accuracy, then the convergence and error estimate can be guanranteed \cite{huang2017quadrature}. However, it is not true for the sparse grid DG method from our numerical experiments. This indicates that the standard DG method is more stable than the sparse grid DG method in this sense. We also remark that, since the interpolation operator introduced here is global but not local, the approach in \cite{huang2017quadrature} would probably fail to obtain the rigorous error estimate here. We will leave the detailed analysis as future work.

\subsection{Adaptivity}\label{subsec:adapt}

In this section, we review the adaptive procedure introduced in \cite{bokanowski2013adaptive,guo2017adaptive} to determine the space $\bV^k.$ The method is very similar to those in \cite{bokanowski2013adaptive,guo2017adaptive}, except that two sets of basis functions are involved and they are adaptively chosen at the same time.

In the adaptive DG algorithm, we specify the maximum mesh level $N$ and an accuracy threshold $\epsilon>0$. 
The same adaptive multiresolution projection method in \cite{guo2017adaptive} is applied here as the numerical initial condition for DG schemes. The error indicator using $L^2$ norm is used. The details are omitted  and we refer readers to Algorithm 1 in \cite{guo2017adaptive}.

The scheme is implemented by hash table as the underlying data structure. We now introduce the concepts of child, parent and leaf elements. If an element $V_{\bl'}^{\bj'}$ with $\abs{\bl'}_{\infty}\le N$ satisfies the condition that there exists an integer $m$ such that $1\le m\le d$ and $\bl'=\bl+\textbf{e}_m$, where $\textbf{e}_m$ denotes the unit vector in the $x_m$ direction, and the support of $V_{\bl'}^{\bj'}$ is within that of $V_{\bl}^{\bj}$, then $V_{\bl'}^{\bj'}$ is called a child element of $V_{\bl}^{\bj}$. Accordingly, element $V_{\bl}^{\bj}$ is called a parent element of $V_{\bl'}^{\bj'}$. If an element does not have its child element in the hash table, then we call it a leaf element.

The time evolution consists of four steps. 
The first step is the prediction step, which means given the hash table $H$ that stores the numerical solution $u_h$ at time step $t^n$ and the associated leaf table $L$, we need to predict the location where the details becomes significant at the next time step $t^{n+1}$, then add more elements in order  to capture the fine structures.  
We solve for $u_h\in \bV_{N,H}^k$ from $t^n$ to $t^{n+1}$ using a cheap solver, e.g. the forward Euler discretization.  Here, the interpolation operator $\mathcal{I}[\cdot]$ is determined by accuracy requirement, and has the same multiresolution structure as determined by the  hash table $H$ corresponding to the numerical solution $u_h$. The predicted solution at $t^{n+1}$ is denoted by $u_h^{(p)}$. Note that to save cost, that the artificial viscosity term as introduced in Section \ref{sec:viscosity} does not need to be included in the prediction step.

The second step is the refinement step according to the predicted solution $u_h^{(p)}$. We traverse the hash table $H$ and if an element $V_\bl^\bj$ satisfies the refinement criteria
\begin{equation}\label{eq:crit}
    ( \sum_{\mathbf{1}\leq\bi\leq \bk+\mathbf{1}} |u^\bj_{\bi,\bl}|^2  )^{1/2} \ge \epsilon,
\end{equation}
where $u^\bj_{\bi,\bl}$ denotes the hierarchical coefficient corresponding to the basis $v^\bj_{\bi,\bl}(\bx),$ i.e. $u^\bj_{\bi,\bl}= \int_{\Omega}u(\bx)v^\bj_{\bi,\bl}(\bx)d\bx.$
\eqref{eq:crit} indicates that such an element becomes significant at the next time step, then we need to refine the mesh by adding its children elements to $H$. The detailed procedure is described as follows. For a child element $V_{\bl'}^{\bj'}$ of $V_\bl^\bj$, if it has been already added to $H$, i.e. $V_{\bl'}^{\bj'}\in H$, we do nothing; if not, we add the element $V_{\bl'}^{\bj'}$ to $H$ {  and set the} associated detail coefficients $u^{\bj'}_{\bi,\bl'}=0,\,\mathbf{1}\leq\bi\leq\bk+\mathbf{1}$. Moreover, we need to increase the number of children by one for all elements that has $V_{\bl'}^{\bj'}$ as its child element and remove the parent elements of $V_{\bl'}^{\bj'}$ from the leaf table   if they have been added. Finally, we obtain a larger hash table $H^{(p)}$ and the associated approximation space $\bV_{N,H^{(p)}}^k$ and the leaf table $L^{(p)}$.

Then,  based on the updated hash table $H^{(p)}$, we   evolve the numerical solution by the DG formulation with space $\bV_{N,H^{(p)}}^k$. Namely, we solve for  $\bV_{N,H^{(p)}}^k$ from $t^n$ to $t^{n+1}$, to generate the precoarsened solution $\tilde{u}_h^{n+1}$, by using the the accurate solver with artificial viscosity in Section \ref{sec:viscosity}. Here, {  the interpolation operator should be determined} by the updated hash table $H^{(p)}$. Note that in the artificial viscosity $\nu=\nu(u_h)$ we fix $u_h$ to be $u_h^n$ such that the matrix for {  the diffusion term only needs to be resembled once} in each time step.

The last step   is to coarsen  by removing elements that become insignificant at time level $t^{n+1}.$  The hash table $H^{(p)}$ that stores the numerical solution $\tilde{u}_h^{n+1}$  is recursively coarsened by the following procedure. 
The leaf table $L^{(p)}$ is traversed, and if an element $V_\bl^\bj\in L^{(p)}$ satisfies the coarsening criterion
\begin{equation}\label{eq:l2_c}
    (\sum_{\mathbf{1}\leq\bi\leq\bk+\mathbf{1}}|u^\bj_{\bi,\bl}|^2)^{\frac12}<\eta,
\end{equation}
where $\eta$ is a prescribed error constant, then we remove the element from both table $L^{(p)}$ and $H^{(p)}$, and set the associated coefficients $u^{\bj'}_{\bi,\bl'}=0,\,\mathbf{1}\leq\bi\leq\bk+\mathbf{1}$. For each of its parent elements in table $H^{(p)}$, we decrease the number of children by one. If the number becomes zero, i.e, the element has no child any more, then it is added to the leaf table $L^{(p)}$ accordingly. Repeat the coarsening procedure until no element can be removed from the table $L^{(p)}$. By removing only the leaf element at each time, we avoid generating ``holes" in the hash table. The output of this coarsening procedure are the updated hash table and leaf table, denoted by $H$ and $L$ respectively, and the compressed numerical solution $u_h^{n+1} \in \bV_{N,H}^k$. In practice, $\eta$ is chosen to be smaller than $\varepsilon$ for safety. In the simulations presented in this paper, we use $\eta = \varepsilon/10$.

	{  For smooth solutions, the adaptive grids automatically reduce to the sparse grid methods as shown in \cite{guo2017adaptive}. We also comment that the adaptive time evolution procedure is stable if the high order RK-DG procedure is stable. This is because
(1) in the prediction step, no change to numerical solution is made; (2) in the refinement step, we only add ``zero" to the solution at $t^n$, so it does not add energy; (3) in the evolution step, we use a stable RK-DG methods; (4) in the coarsening step, we remove the coefficients when measured in orthogonal multiwavelet bases, so the $L^2$ energy is guaranteed to decay.   }

{ 

\subsection{Fast computations of the residual}\label{subsec:fast}
We  now describe the details of the computation of the right hand side of DG weak formulation \eqref{eq:DG-semi-interp}. 
This is important because the multiwavelet bases are global, and the evaluation of the residual yields denser matrix than those obtained by standard local bases. Our work extends the  fast matrix-vector multiplication in \cite{shen2010efficient,zeiser2011fast}  to adaptive index set.
Consider   matrix-vector multiplication in multidimensions in the following form:
\begin{equation}\label{eq:LU-original}
	f_{\bm{n}} = \sum_{H(\bm{n}')\le 0} f'_{\bm{n}'} t_{n_1', n_1}^{(1)} t_{n_2',n_2}^{(2)}\cdots t_{n_d',n_d}^{(d)}, \quad H(\bm{n})\le 0,
\end{equation}
where $\bm{n}=({n}_1,{n}_2,\dots,{n}_d)$ and $\bm{n}'=({n}_1',{n}_2',\dots,{n}_d')$ can be thought of as  the level of the mesh, and $t_{n_1', n_1}^{(i)}=T^{(i)}_{n_1', n_1}$ represents the calculations in the $i$-th dimension. It is assumed that the constraint function $H=H(\bm{n}')=H({n}_1',{n}_2',\dots,{n}_d')$ is non-decreasing with respect to each variable. This holds true for sparse grid (by taking $H(\bm{n}')=|\bm{n}'|_1$) and also for adaptive grid considered in this work.

One can compute the sum \eqref{eq:LU-original} dimension-by-dimension, i.e. first perform the transformation in the $x_1$ dimension:
\begin{equation}\label{eq:adapt-x1}
    g^{(1)}_{(n_1,n_2',\dots,n_d')} = \sum_{H(n_1',n_2',\dots,n_d')\le 0} f'_{(n_1',n_2',\dots,n_d')} t_{n_1', n_1}^{(1)},
\end{equation}
and then in the $x_2$ dimension:
\begin{equation}\label{eq:adapt-x2}
    g^{(2)}_{(n_1,n_2,n_3'\dots,n_d')} = \sum_{H(n_1,n_2',\dots,n_d')\le 0} g^{(1)}_{(n_1,n_2',\dots,n_d')} t_{n_2', n_2}^{(2)},
\end{equation}
and all the way up to $x_d$ dimension:
\begin{equation}\label{eq:adapt-xd}
    f_{(n_1,n_2,n_3\dots,n_d)} = \sum_{H(n_1,n_2,\dots,n_{d-1},n_d')\le 0} g^{(d-1)}_{(n_1,n_2,\dots,n_{d-1},n_d')} t_{n_d', n_d}^{(d)},
\end{equation}
It can be proved that the   \eqref{eq:adapt-x1}-\eqref{eq:adapt-xd} is equivalent to the original summation \eqref{eq:LU-original}, if assuming that, for some integer $1\le k\le d$, $T^{(i)}$ for $i=1,\dots,k-1$ are strictly block lower triangular and $T^{(i)}$ for $i=k+1,\dots,d$ are block upper triangular (or $T^{(i)}$ for $i=1,\dots,k-1$ are block lower triangular and $T^{(i)}$ for $i=k+1,\dots,d$ are strictly block upper triangular) \cite{shen2010efficient}.

When such properties for $T$ matrices are not true, one can perform  $L+U$ split and   \eqref{eq:LU-original} becomes:
\begin{equation}
    f_{\bm{n}} = \sum_{H(\bm{n}')\le 0} f'_{\bm{n}'} (l_{n_1', n_1}^{(1)}+u_{n_1', n_1}^{(1)})(l_{n_1', n_1}^{(1)}+u_{n_1', n_1}^{(1)})\cdots (l_{n_{d-1}', n_{d-1}}^{(d-1)}+u_{n_{d-1}', n_{d-1}}^{(d-1)})t_{n_d',n_d}^{(d)},
\end{equation}
where there are totally $2^{d-1}$ terms. For each term, we can perform the dimension-by-dimension matrix-vector product. 
The overall computational cost   is $\mathcal{O}(2^{d-1} \cdot DoF \cdot N)$   if the cost of one-dimensional transform is log-linear, i.e., $\mathcal{O}(\mathcal{N}\log \mathcal{N})$  where $\mathcal{N}$ denotes the DoF in one-dimension \cite{shen2010efficient}. This assumption holds true for our scheme.

Now, we return to the description of the implementation of  \eqref{eq:DG-semi-interp}. The computations are done using the following steps with repeated application of the fast matrix-vector product described above. 
We denote the adaptive numerical solution   by
\begin{equation}\label{eq:alpert-basis}
	u_h(\bx) = \sum_{{(\bl,\bj)\in G,  \mathbf{1}\le \bi\le \bk+\mathbf{1}}} c^\bj_{\bi,\bl} v^\bj_{\bi,\bl}(\bx),
\end{equation}
where $v^\bj_{\bi,\bl}(\bx)$ is the Alperts' multiwavelets in multidimensions defined in \eqref{eq:multidim-basis} and the set $G$ collects the index of all active elements. The active index set of the interpolatory multiwavelets is also equal to G. In particular, the adaptive interpolation function space is denoted by 
\begin{equation}\label{eq:interp-basis}
	\bV^M = \{ \psi_{\bi,\bl,\bn}^{\bj}: (\bl,\bj)\in G, \mathbf{1}\le \bi\le \bk+\mathbf{1}, \mathbf{1}\le \bn\le \bP+\mathbf{1} \}.
\end{equation}
The corresponding interpolation points are
\begin{equation}\label{eq:interp-pt}
	Q = \{ \tilde{\bx}_{\bi,\bl}^{\bj}=(\tilde{x}_{i_1,l_1}^{j_1}, \tilde{x}_{i_2,l_2}^{j_2}, \cdots, \tilde{x}_{i_d,l_d}^{j_d}):\mathbf{1}\le \bi\le \bk+\mathbf{1}, (\bl,\bj)\in G \}.
\end{equation}

The first step is to obtain  the function and its derivative value at the interpolation points \eqref{eq:interp-pt} from the coefficients $ c^\bj_{\bi,\bl}.$ We denote
\begin{equation}\label{eq:interp-up}
\begin{aligned}
	(u_p)_{\bi,\bl,\bn}^{\bj} := \frac{\partial u_h}{\partial \bx^{\bn}}(\tilde{\bx}_{\bi,\bl}^{\bj}) &= \sum_{{(\bl,\bj)\in G,  \mathbf{1}\le \bi\le \bk+\mathbf{1}}} c^\bj_{\bi,\bl} \frac{\partial v^\bj_{\bi,\bl}}{\partial \bx^{\bn}}(\tilde{\bx}_{\bi,\bl}^{\bj}) 	 \\
	&= \sum_{{(\bl,\bj)\in G,  \mathbf{1}\le \bi\le \bk+\mathbf{1}}} c^\bj_{\bi,\bl} \prod_{m=1}^d \frac{dv^{j_m}_{i_m,l_m}}{dx^{n_m}}(\tilde{x}_{i_m,l_m}^{j_m}), 
\end{aligned}
\end{equation}
Here, the values of Alperts' basis functions in 1D and their derivatives at all the interpolation points in 1D should be computed and stored before the time evolution. The fast matrix-vector multiplication described above is applied to evaluate this summation \eqref{eq:interp-up}.

The second step is to calculate the value of $f(u_h)$ and its derivative at all the interpolation points \eqref{eq:interp-pt}, which is denoted by $\{ (f_p)_{\bi,\bl,\bn}^{\bj} \}$, by using $\{ (u_p)_{\bi,\bl,\bn}^{\bj} \}$:
\begin{equation*}
	(f_p)_{\bi,\bl,\bn}^{\bj} := \frac{\partial f(u_h)}{\partial \bx^{\bn}}(\tilde{\bx}_{\bi,\bl}^{\bj}).
\end{equation*}
This can be obtained by simply using the chain rule.
Afterwards, we can  transform the point values $\{ (f_p)_{\bi,\bl,\bn}^{\bj} \}$ to the coefficient of the interpolation basis $ \tilde{c}^\bj_{\bi,\bl,\bn}$  in 
\begin{equation}\label{eq:interp-fu}
	\mathcal{I}[f(u_h)] = \sum_{\substack{(\bl,\bj)\in G, \\ \mathbf{1}\le \bi\le \bk+\mathbf{1}, \\ \mathbf{1}\le \bn\le \bP+\mathbf{1} }} \tilde{c}^\bj_{\bi,\bl,\bn} \psi^\bj_{\bi,\bl,\bn}(\bx)
\end{equation}
by applying the fast algorithm in \cite{tao2019collocation}.

Now, the terms in \eqref{eq:DG-semi-interp} can be readily computed. For example, for the volume integral $\sum_{K\in T_h}\int_{K}\mathcal{I}[f(u_h)]\cdot\nabla v_h d\bx,$ 
let the test function be $v_h = v^{\bj'}_{\bi',\bl'}(\bx),$ we have 
\begin{equation}
	\begin{aligned}
		\sum_{K\in T_h}\int_{K}\mathcal{I}[f(u_h)]\cdot\nabla v_h d\bx 
		&= \int_{\Omega} \mathcal{I}[f(u_h)]\cdot\nabla v_h d\bx   \\
		&= \int_{\Omega} \sum_{\substack{(\bl,\bj)\in G, \\ \mathbf{1}\le \bi\le \bk+\mathbf{1}, \\ \mathbf{1}\le \bn\le \bP+\mathbf{1} }} \tilde{c}^\bj_{\bi,\bl,\bn} \psi^\bj_{\bi,\bl,\bn}\cdot\nabla v^{\bj'}_{\bi',\bl'} d\bx   \\
		&= \sum_{\substack{(\bl,\bj)\in G, \\ \mathbf{1}\le \bi\le \bk+\mathbf{1}, \\ \mathbf{1}\le \bn\le \bP+\mathbf{1} }} \tilde{c}^\bj_{\bi,\bl,\bn} \int_{\Omega} \psi^\bj_{\bi,\bl,\bn}\cdot\nabla v^{\bj'}_{\bi',\bl'} d\bx   \\
		&= \sum_{\substack{(\bl,\bj)\in G, \\ \mathbf{1}\le \bi\le \bk+\mathbf{1}, \\ \mathbf{1}\le \bn\le \bP+\mathbf{1} }} \tilde{c}^\bj_{\bi,\bl,\bn} \prod_{m=1}^d \int_0^1 \psi^{j_m}_{{i_m},{l_m},{n_m}} \frac{d}{d x_m}{v^{j'_m}_{i'_m,l'_m}} dx
	\end{aligned}
\end{equation}
This can be again treated by the fast matrix-vector multiplication algorithm. The 1D values $\int_0^1 \psi^{j_m}_{{i_m},{l_m},{n_m}} \frac{d}{dx}{v^{j'_m}_{i'_m,l'_m}} dx$  should be precomputed and stored before the time evolution starts.
The computation of the edge integrals   $ \sum_{e\in\Gamma_h}\int_{e}\mathcal{I}[\widehat{f\cdot n_K}(u_h)]v_h ds $ also follows the same approach as that over the volume. The details are omitted here for brevity.
}

\subsection{Artificial viscosity} 
\label{sec:viscosity}

{  For capturing shocks}, we add artificial viscosity following the approach in \cite{bassi1995viscosity} and arrive at the semi-discrete formulation
\begin{align}\label{eq:semi-DG-viscosity}
    \sum_{K\in T_h}\int_{K} (u_h)_t v_h d\bx - \sum_{K\in T_h}\int_{K}\mathcal{I}[f(u_h)]\cdot\nabla v_h d\bx + \sum_{e\in \Gamma_h}\int_{e}\mathcal{I}[\widehat{f\cdot n_K}(u_h)]v_h ds \notag\\
     - \sum_{K\in T_h}\int_K \nu(u_h)\nabla u_h\cdot\nabla v_h d\bx = 0.
\end{align}
where $\nu=\nu(u_h)\ge0$ is the artificial viscosity. The artificial viscosity is piecewise constant in each element and depends on $u_h$. {  Since the sharp gradient and shock will only appear in the leaf element \cite{hovhannisyan2014scalar},} the artificial viscosity is only imposed in the leaf element and determined in the following approach:
$$ \nu=\left\{
\begin{aligned}
& 0, \quad  &\text{if}\quad s_e \le s_0 + \kappa, \\
& \nu_0 h, \quad &\text{otherwise}.
\end{aligned}
\right.
$$
where $\nu_0>0$ and $\kappa$ are constants chosen empirically, see \cite{Discacciati2018thesis, hesthaven2019viscosity} for discussions on standard DG methods. {  In this paper, we use $\nu_0=2$ and $\kappa=0$.}  $s_e$ and $s_0$ are defined as
\begin{equation}
    s_e = \log_{10}(\sum_{\mathbf{1}\leq\bi\leq\bk+\mathbf{1}}|u^\bj_{\bi,\bl}|^2)^{\frac12}, \quad s_0 = \log_{10}(2^{-(k+\frac12)|\bl|_1}).
\end{equation}
In the regions where the solutions are smooth, $s_e$ should be the same order as $s_0$ by the estimate \eqref{eq:smooth-coeff-l2-norm}. If the solution is discontinuous, $s_e$ should be much larger than $s_0$.  

\begin{rem}
    There are still many problems to be explored on the artificial viscosity. The first one is the specific form of the artificial viscosity term. Here, for simplicity, we only add an artificial viscosity term $\int \nu(u_h) \nabla u_h\cdot \nabla v_h$ in \eqref{eq:semi-DG-viscosity}. One may also add a physical diffusion term $\nabla\cdot(\nabla \nu(u)u)$ and {  then discretize it} using local DG \cite{persson2006shock} or interior penalty DG \cite{klockner2011viscous}. The second issue is how to choose the optimal parameters $\kappa$ and $\nu_0$ in the artificial viscosity to obtain a sharp shock profile. The artificial neural network introduced in \cite{ray2018troublecell,hesthaven2019viscosity} might be helpful for this problem. We will explore these subjects in future work.
\end{rem}

The diffusion coefficient $\nu(u_h)$ is of order $\mathcal{O}(h)$ for trouble cells and zero for normal cells. Thus, the explicit time integration in both convection and diffusion terms in \eqref{eq:semi-DG-viscosity} will yield  CFL condition $\Delta t=\mathcal{O}(h)$. For hyperbolic problems $u_t=u_x$ with DG discretizations using polynomials of degree $k$ and upwind numerical flux and a $(k+1)$ stage explicit RK method of order $(k+1)$, the CFL constant is around $\frac{1}{2k+1}$ \cite{cockburn2001review}. However, for solving diffusion equation $u_t=u_{xx}$ with local DG discretization with polynomials and alternating numerical flux, the CFL constant is around 0.0555 for $k=1$, 0.0169 for $k=2$, 0.0063 for $k=3$, and 0.003 for $k=4$, if coupled with explicit Runge-Kutta methods of the corresponding order\footnote{The CFL constants are provided by Chi-Wang Shu from Brown University in personal communications.}, which is much smaller   than the CFL constant for convection terms, especially for polynomials of high degrees. If the alternating numerical flux is replaced by the central flux for the diffusion equation, the CFL constant is slightly larger but still much smaller than the CFL constant for the convection part: 0.125 for $k=1$, 0.0384 for $k=2$, 0.0158 for $k=3$ and 0.0083 for $k=4$. 

To obtain better computational efficiency, we avoid   explicit time integrations and apply the IMEX time discretizations where the convection term is treated explicitly and the diffusion term implicitly. Here, we only present the third-order IMEX method introduced in \cite{pareschi2005imex}, which will be coupled with the DG space of quadratic polynomials. The explicit part is the same with the explicit third-order strong stability preserving (SSP) Runge-Kutta method \cite{shu1988jcp} and the implicit part has four stages. To be precise, for the ODE systems:
\begin{equation}\label{eq:IMEX-ODE}
    \frac{dU}{dt} = F(U) + G(U),
\end{equation}
where $F(U)$ denotes the non-stiff term (convection parts) and $G(U)$ the stiff term (diffusion parts). The IMEX scheme for \eqref{eq:IMEX-ODE} reads as
\begin{subequations}\label{eq:IMEX-scheme}
\begin{align}
	U^{(i)} &= U^n + \Delta t\sum_{j=1}^{i-1} \tilde{a}_{ij} F(U^{(j)}) + \Delta t \sum_{j=1}^{i} {a_{ij}} G(u^{(j)}), \quad i=1,\dots, \nu, \\
	U^{n+1} &= U^n + \Delta t\sum_{i=1}^{\nu} \tilde{w}_{i} F(U^{(i)}) + \Delta t \sum_{i=1}^{\nu} {w_i} G(u^{(i)}),
\end{align}
\end{subequations}
with the stage $\nu=4$ and the parameters
\begin{align*}
    & \tilde{a}_{32} = 1, \, \tilde{a}_{42}=\tilde{a}_{43}=\frac{1}{4}, \, a_{11}=a_{22}=a_{33}=a_{44}=\alpha, \, a_{21}=-\alpha, \, a_{32}=1-\alpha, \\
    & a_{41}=\beta, \, a_{42}=\eta, \, a_{43}=\frac{1}{2}-\beta-\eta-\alpha, \, \tilde{w}_2=\tilde{w}_3={w}_2={w}_3=\frac{1}{6}, \, \tilde{w}_4={w}_4=\frac{2}{3},
\end{align*}
and
\begin{equation*}
    \alpha=0.24169426078821, \quad \beta = 0.06042356519705, \quad \eta = 0.12915286960590.
\end{equation*}
The other parameters not listed above are zero.

By using the IMEX time integrator, the time step restriction remains the same as determined by the convection term. Note that the artificial viscosity $\nu=\nu(u)$ is determined by $u_h^n$ and will keep unchanged in the middle stages of time evolution from $t^n$ to $t^{n+1}$. Therefore, the matrix for {  the diffusion term only needs to be assembled once} in each time step. Also, we only need to solve a linear system in which the coefficient matrix is symmetric positive definite and also sparse (there exist only a small portion of elements with non-zero viscosity). In the computation, we apply the conjugate gradient method to solve this linear system. We also remark that, for smooth solutions, this scheme will reduce to the explicit time integrations when coupled with the semi-discrete DG scheme with artificial viscosity \eqref{eq:semi-DG-viscosity}, since the artificial viscosity will automatically vanish and then IMEX scheme \eqref{eq:IMEX-scheme} reduces to the third-order SSP RK method. {  For each element $V_l^j$ that satisfies the criteria, we compute the matrix corresponding to the term $\nu_0 h\int_{V_l^j} u_x v_x$ by using the undirectional approach. Then, by making a summation over all the elements that have artificial viscosity, we obtain the corresponding matrix term.}

\section{Numerical results}

In this section, we perform numerical experiments to validate the accuracy and robustness of our scheme. The computational domain is $[0,1]$ for 1D and $[0,1]^2$ for 2D. Periodic boundary condition is imposed. When testing accuracy for smooth solutions, we apply the TVD Runge-Kutta time discretizations \cite{shu1988jcp}: second-order RK method for the piecewise linear finite element space ($k=1$) and third-order RK method for the quadratic ($k=2$) and cubic ($k=3$) finite element space. When testing the capability for capturing discontinuous solutions, we use the quadratic finite element space ($k=2$) coupled with the third-order IMEX time discretizations \eqref{eq:IMEX-scheme}. {  The DoF refers to the number of Alperts' multiwavelets basis functions in the adaptive grids, i.e. $\textrm{dim}(\bV^k)$.}

\begin{examp}[1D linear advection with constant coefficient]\label{exam:linear-1d}
In this example, we consider the 1D linear advection equation with constant coefficient:
    \begin{equation}
        u_t + u_x  = 0,
    \end{equation}
    with the initial value $u(x,0) = u_0(x)$.   Since the equation is linear, the interpolation operator is not needed. We focus on a non-smooth initial condition to test the effectiveness of adaptivity and the artificial viscosity. The initial condition is chosen as:
    \begin{equation}
         u_0(x)=\left\{
            \begin{aligned}
            & 1, \quad 0.23 < x < 0.56, \\
            & 0, \quad \textrm{otherwise}.
            \end{aligned}
            \right.
     \end{equation}
    We set $N=8$ and $\epsilon=10^{-5}$. The solutions with and without artificial viscosity at $t=3$ are presented in Fig. \ref{fig:linear1d-solution-profile}. We observe that the multiresolution DG method without artificial viscosity as in \cite{guo2017adaptive} can also produce well-resolved solution. With the artificial viscosity, the oscillations are suppressed.
    \begin{figure}
        \centering
        \subfigure[without artificial viscosity]{
        \begin{minipage}[b]{0.46\textwidth}
        \includegraphics[width=1\textwidth]{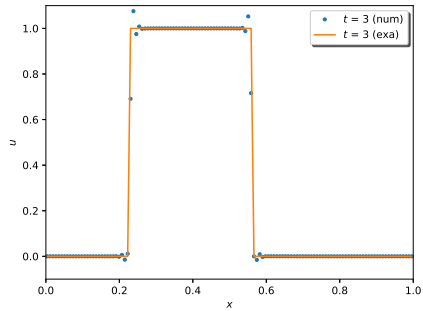}    
        \end{minipage}
        }
        \subfigure[with artificial viscosity]{
        \begin{minipage}[b]{0.46\textwidth}    
        \includegraphics[width=1\textwidth]{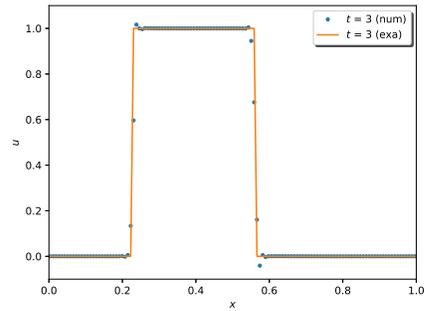}
        \end{minipage}
        }
        \caption{Example \ref{exam:linear-1d}: 1D linear advection with constant coefficient at $t=3$. Left: without artificial viscosity; right: with artificial viscosity. The solid lines are the exact solution and the symbols are the numerical solutions.}
        \label{fig:linear1d-solution-profile}
    \end{figure}

    In Fig. \ref{fig:linear1d-dof}, the degrees of freedom and the errors for scheme with and without artificial viscosity are presented. Since the artificial viscosity (diffusion term) can smoothen the solution, the method has fewer degrees of freedom and thus less computational cost. {  It is also observed that the error with artificial viscosity are comparable to that without artificial viscosity. Note that for this example, the full grid method offers excellent accuracy in the smooth region because the solution is piecewise constant. This is in general not true, see for example Fig. \ref{fig:burgers2d-full}.}
    \begin{figure}
        \centering
        \subfigure[time history of degrees of freedom]{
        \begin{minipage}[b]{0.46\textwidth}
        \includegraphics[width=1\textwidth]{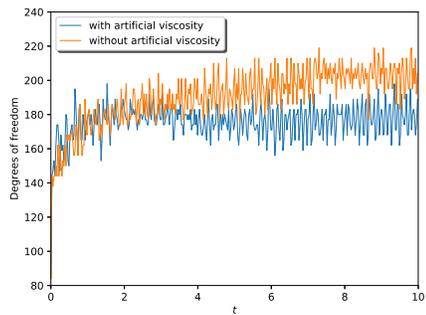}    
        \end{minipage}
        }
        \subfigure[{  error between exact solution and numerical solution at $t=3$}]{
        \begin{minipage}[b]{0.46\textwidth}    
        \includegraphics[width=1\textwidth]{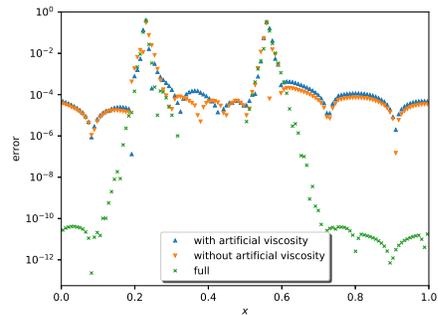}
        \end{minipage}
        }
        \caption{Example \ref{exam:linear-1d}: 1D linear advection with constant coefficient at $t=3$. Left: Time histories of the number of active degrees of freedom with and without artificial viscosity. Right: error between exact solution and numerical solution at $t=3$ with and without artificial viscosity.}
        \label{fig:linear1d-dof}
    \end{figure}
\end{examp}

\begin{examp}[1D Burgers' equation]\label{exam:burgers-1d} In this example, we focus on the 1D Burgers' equation.
    \begin{equation*}
        u_t + \brac{\frac{u^2}{2}}_x = 0,
    \end{equation*}
    with the initial value $u(x,0) = u_0(x)=\sin(2\pi x)+\frac{1}{2}$. 
The shock begins to develop at $t=\frac{1}{2\pi} \approx 0.159$. For this example, we only focus on the non-smooth solution, and defer the accuracy study for smooth solution to the next example in 2D. The adaptive multiresolution DG scheme without artificial viscosity will blow up when the shock develops. We run the code up to time $t=0.2,$ with maximum mesh level is $N=8$. The solutions obtained with artificial viscosity  are shown in Fig. \ref{fig:burgers1d-shock-eps1e3} with $\epsilon=10^{-3}$ and Fig. \ref{fig:burgers1d-shock-eps1e4} with $\epsilon=10^{-4}$. Our scheme can capture the shock very well. We also observe that the leaf elements concentrate near the shock. The artificial viscosity is only imposed on elements near the shock.

    \begin{figure}
        \centering
        \subfigure[solution profile at $t=0.1578$ ]{
        \begin{minipage}[b]{0.46\textwidth}
        \includegraphics[width=1\textwidth]{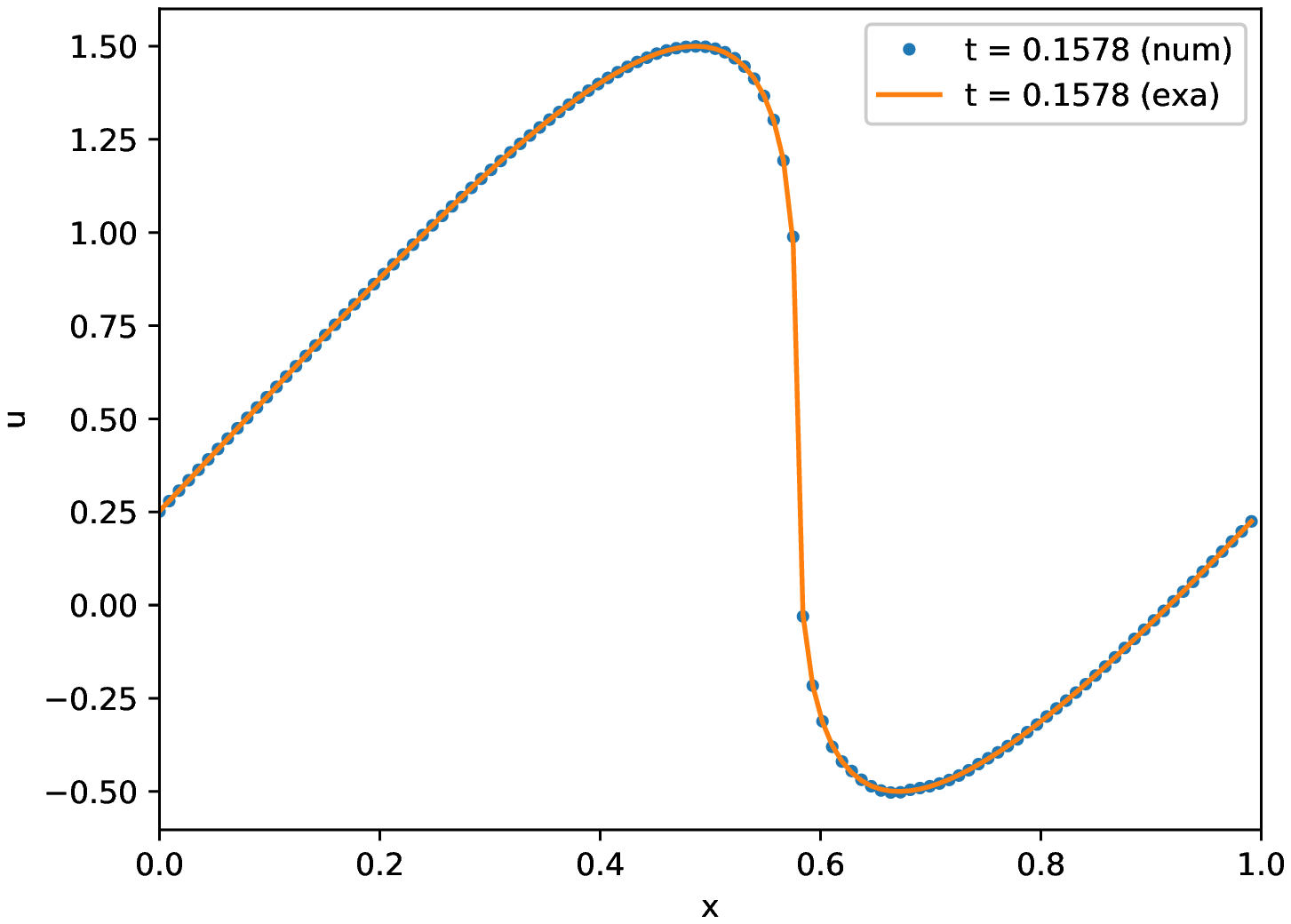}
        \end{minipage}
        }
        \subfigure[supports of active elements and elements with artificial viscosity $t=0.1578$]{
        \begin{minipage}[b]{0.46\textwidth}    
        \includegraphics[width=1\textwidth]{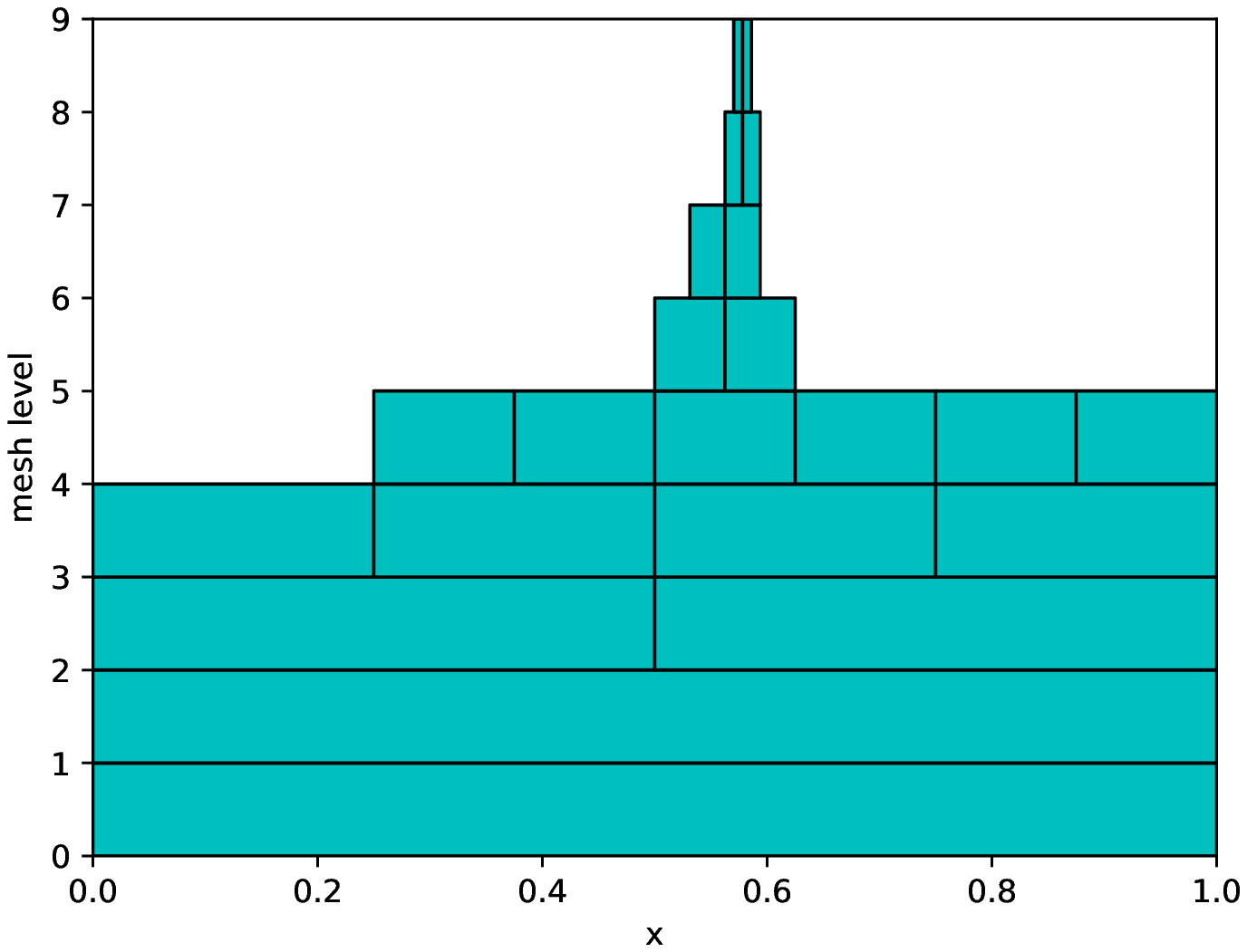}
        \end{minipage}
        }
        \bigskip
        \subfigure[solution profile at $t=0.2$]{
        \begin{minipage}[b]{0.46\textwidth}
        \includegraphics[width=1\textwidth]{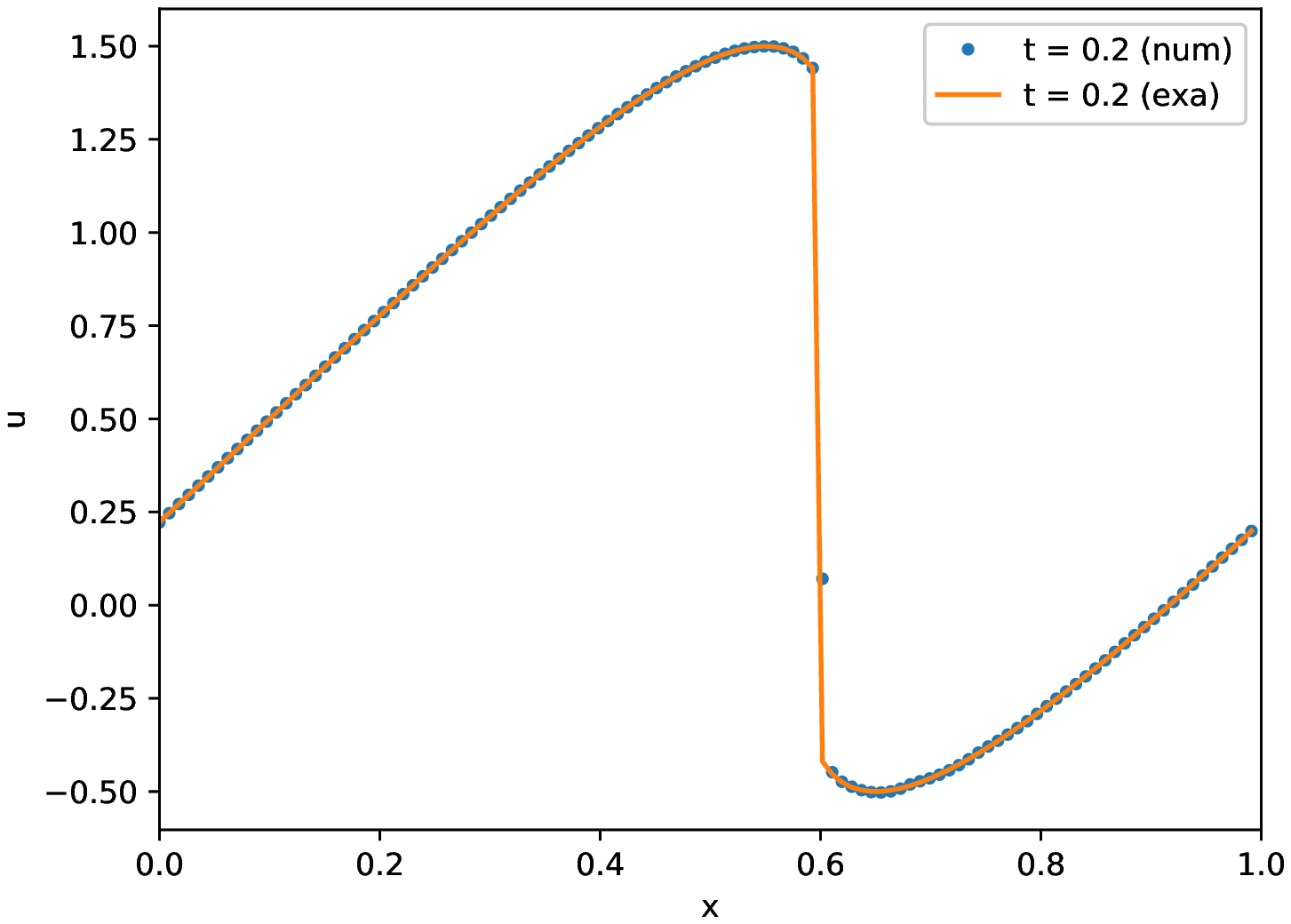}
        \end{minipage}
        }
        \subfigure[supports of active elements and elements with artificial viscosity at $t=0.2$]{
        \begin{minipage}[b]{0.46\textwidth}    
        \includegraphics[width=1\textwidth]{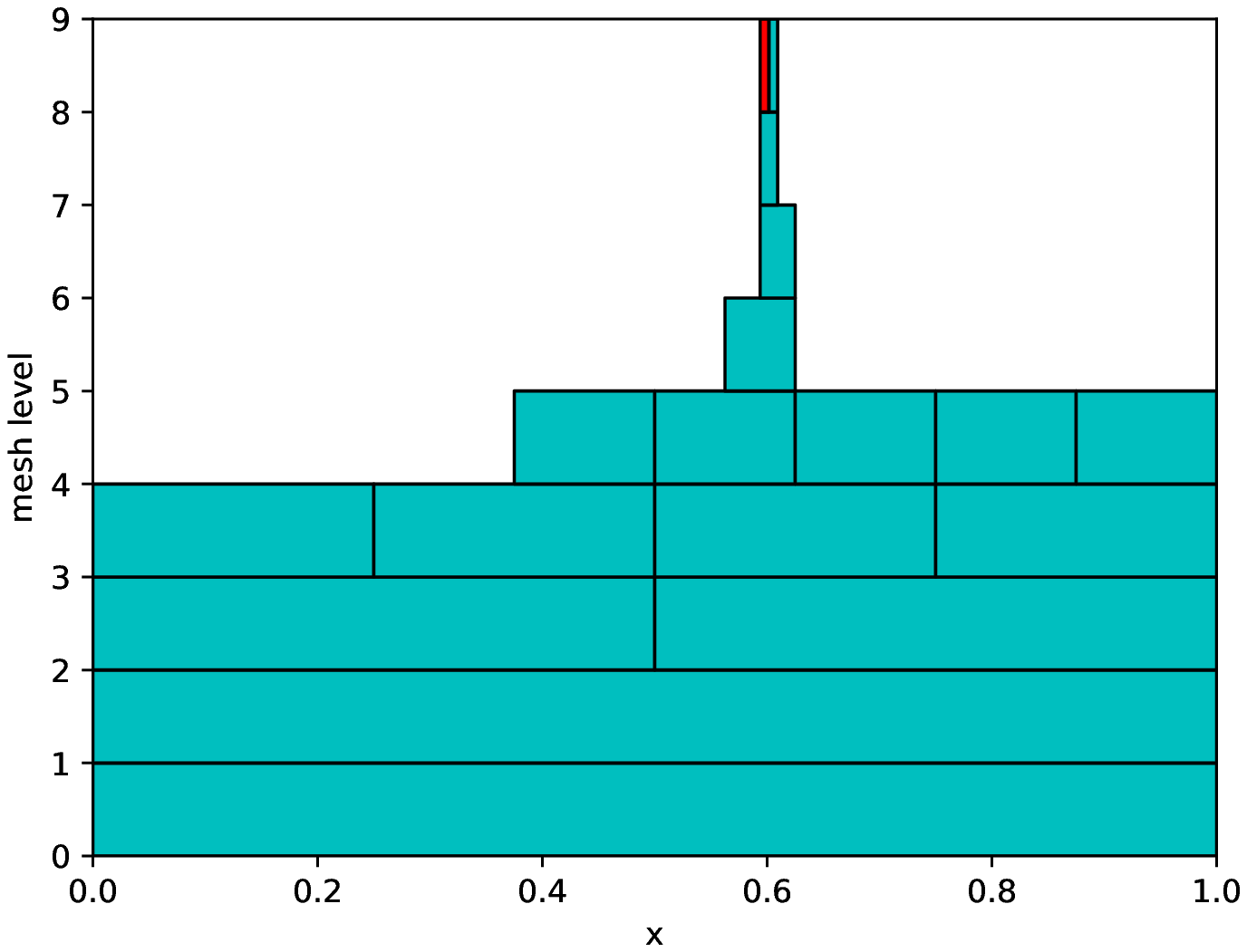}
        \end{minipage}
        }        
        \caption{Example \ref{exam:burgers-1d}: 1D Burgers' equation at $t=0.1578$, 0.1875 and 0.2. $N=8$ and $\epsilon=10^{-3}$. Left: solution profile; right: blue color denotes supports of active elements, and red color denotes elements with non-zero artificial viscosity.}
        \label{fig:burgers1d-shock-eps1e3}
    \end{figure}

    \begin{figure}
        \centering
        \subfigure[solution profile at $t=0.1578$ ]{
        \begin{minipage}[b]{0.46\textwidth}
        \includegraphics[width=1\textwidth]{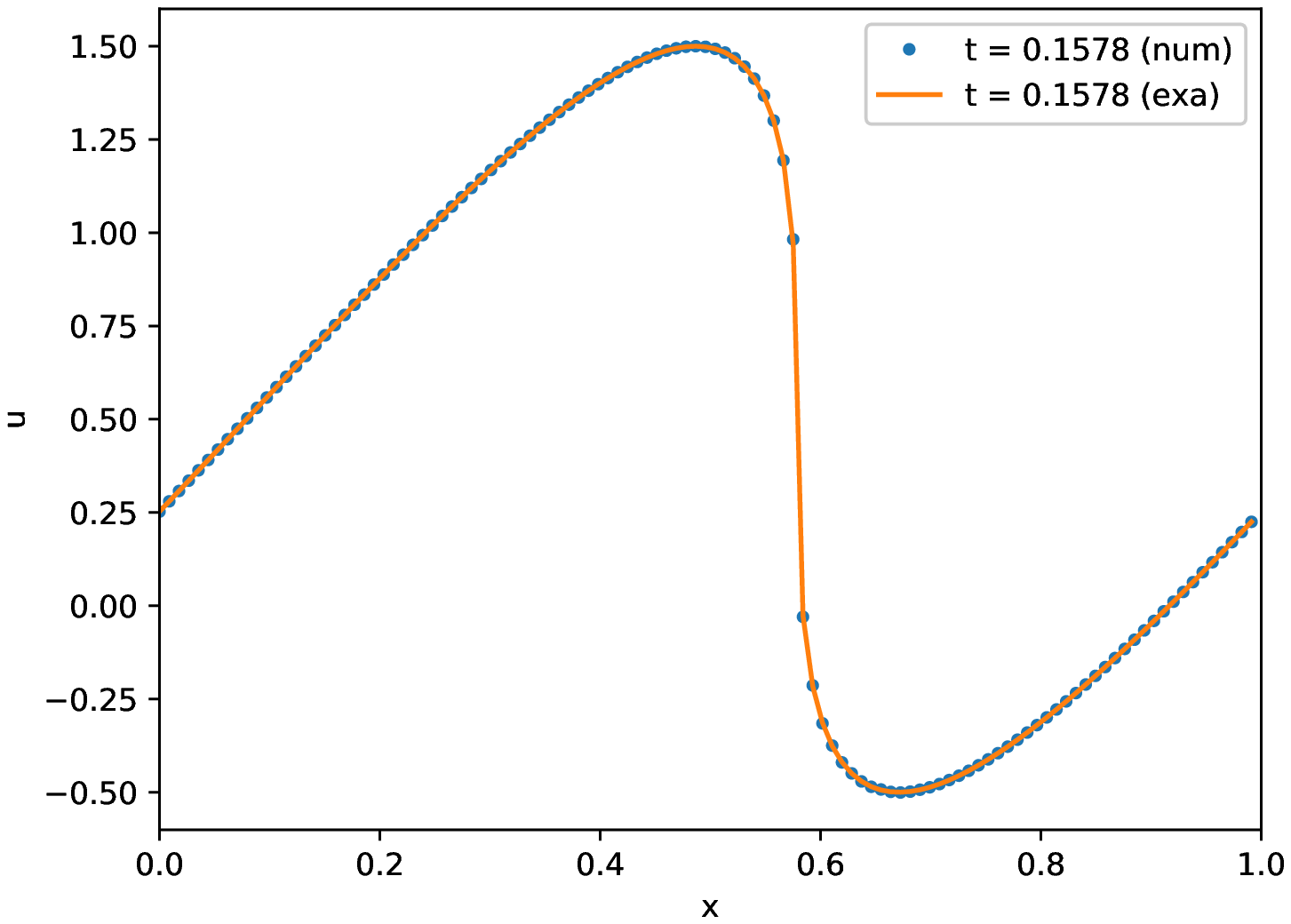}
        \end{minipage}
        }
        \subfigure[supports of active elements and elements with artificial viscosity $t=0.1578$]{
        \begin{minipage}[b]{0.46\textwidth}    
        \includegraphics[width=1\textwidth]{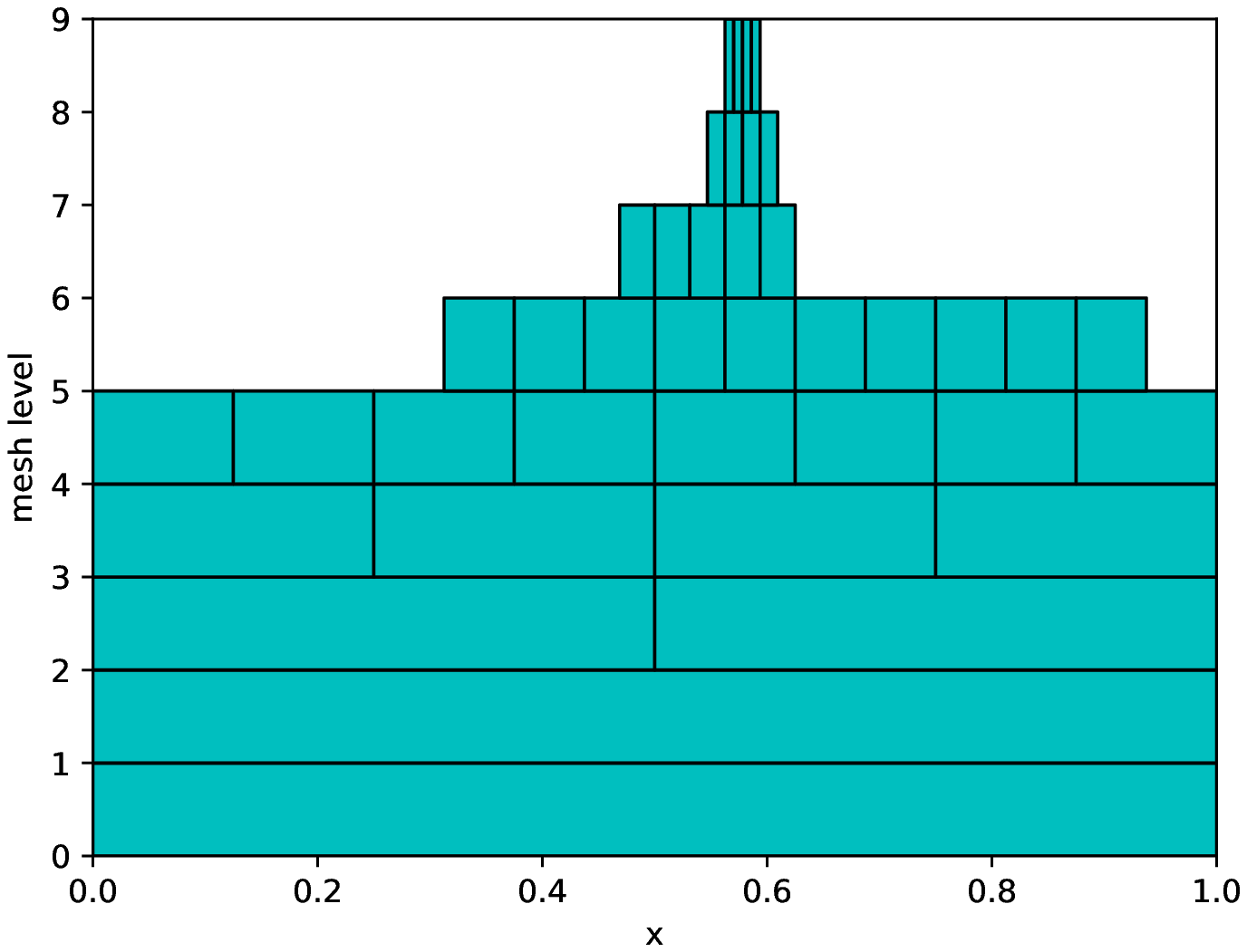}
        \end{minipage}
        }
        \bigskip
        \subfigure[solution profile at $t=0.2$]{
        \begin{minipage}[b]{0.46\textwidth}
        \includegraphics[width=1\textwidth]{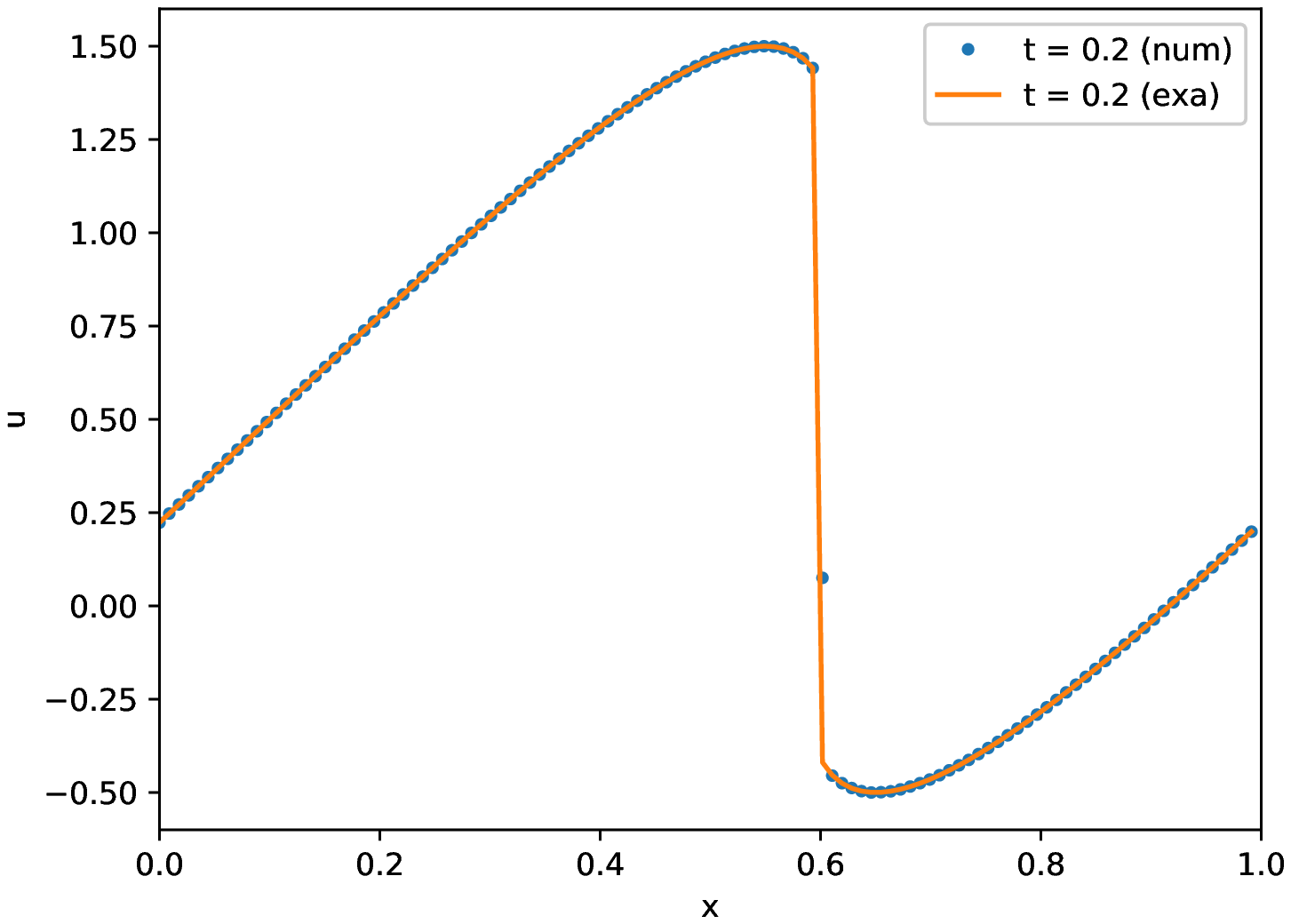}
        \end{minipage}
        }
        \subfigure[supports of active elements and elements with artificial viscosity at $t=0.2$]{
        \begin{minipage}[b]{0.46\textwidth}    
        \includegraphics[width=1\textwidth]{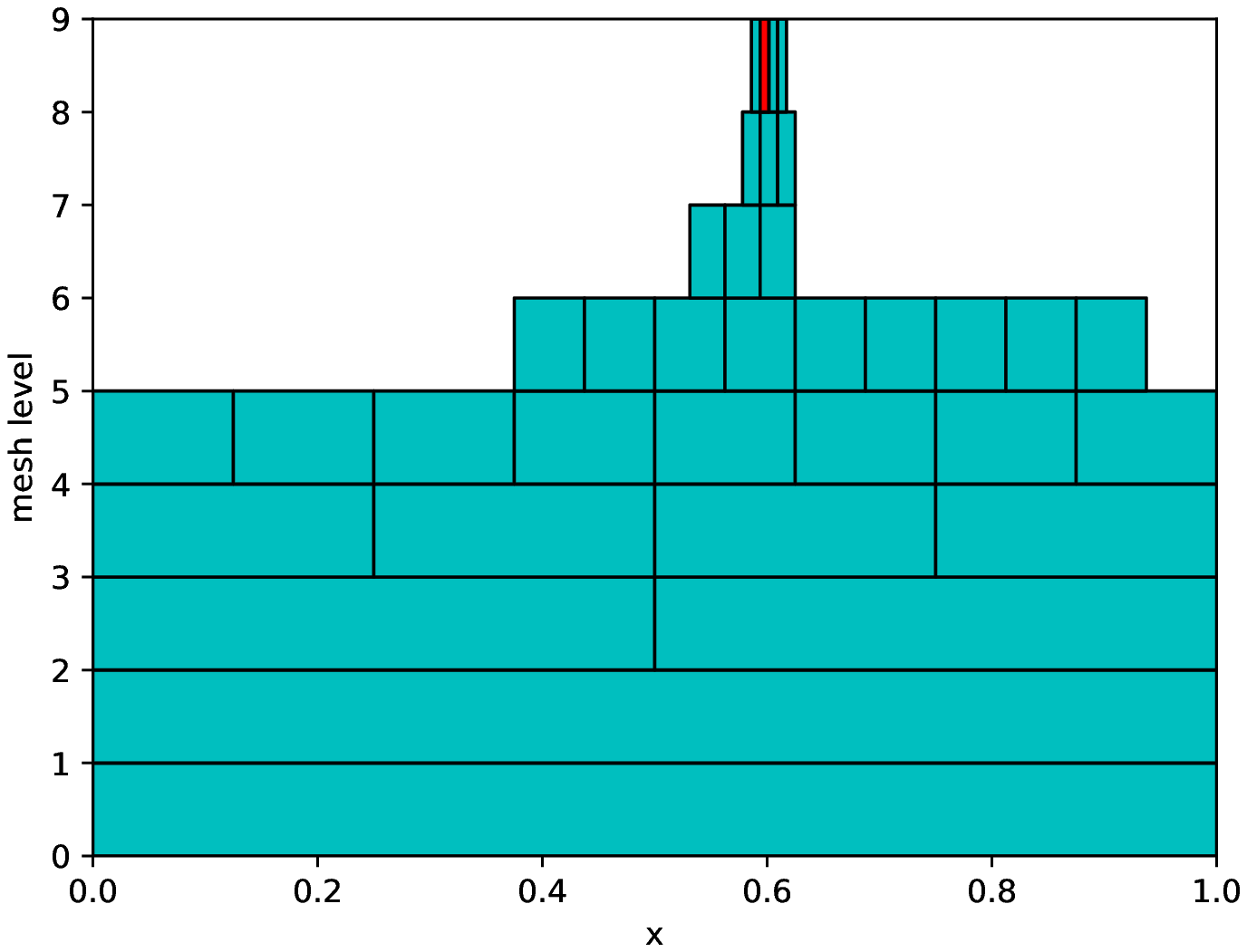}
        \end{minipage}
        }        
        \caption{Example \ref{exam:burgers-1d}: 1D Burgers' equation at $t=0.1578$, 0.1875 and 0.2. $N=8$ and $\epsilon=10^{-4}$. Left: solution profile; right: blue color denotes supports of active elements, and red color denotes elements with non-zero artificial viscosity.}
        \label{fig:burgers1d-shock-eps1e4}
    \end{figure}

\end{examp}

    



\begin{examp}[2D Burgers' equation]\label{exam:burgers-2d}
In this example, we consider the 2D Burgers' equation:
\begin{equation*}
    u_t + \brac{\frac{u^2}{2}}_x + \brac{\frac{u^2}{2}}_y = 0,
\end{equation*}
with the initial value $u=u_0(x,y)=\sin(2\pi(x+y))$.

To study the effect of the interpolation operator, we first test the convergence rates  for smooth solutions without adaptivity and artificial viscosity. The code is run up to $T=0.01$. Table \ref{ex:table-2D-Burgers-Full-Lagr-inner} shows the convergence rate of the standard DG with integrals over elements and edges calculated by Lagrange interpolation techniques. The interpolation points are imposed in the inner domain, see the interpolation points and basis functions in Appendix \ref{subsec:append-lagrange-inner}. Recall, the degree of polynomials for the DG finite element space is denoted by $k$ and the degree of interpolation operator is denoted by $M$. It shows almost full convergence rate, except some order reduction when $k=2$ and $M=2$, similar to previous results in \cite{huang2017quadrature}. Actually, for $k=2$ and $M=4$, the numerical scheme is the same as the   DG scheme in which the integrals are evaluated exactly, since the physical flux for Burgers' equation is a quadratic function. It seems that the convergence rate is almost full order with $k=1$ and $M=1$. Therefore, the condition assumed in Proposition \ref{prop:accurate-interp} may not be necessary for the standard DG method. We also test the accuracy using the Lagrange interpolation in which the interpolation points are at the interface (see the interpolation points and basis functions in Appendix \ref{subsec:append-lagrange-interface}). The results are similar to those in Table \ref{ex:table-2D-Burgers-Full-Lagr-inner}, and thus they are not presented here for saving space. The Hermite interpolations with $k=1,M=3$ and $k=2,M=3$ and $k=3,M=5$ (see the interpolation points and basis functions in Appendix \ref{sec:append-hermite}) are also tested. The same results are observed and are omitted for brevity.
\begin{table}[!hbp]
\centering
\caption{2D Burgers' equation at $T=0.01$, $\Delta t=0.1h$, standard DG, Lagrange interpolation with inner interpolation points.}
\label{ex:table-2D-Burgers-Full-Lagr-inner}
\begin{tabular}{c|c|c|c|c|c|c|c}
  \hline
 poly degrees & $N$ & L$^1$-error & order & L$^2$-error & order & L$^{\infty}$-error & order \\
  \hline
\multirow{4}{4em}{$k$ = 1, $M$ = 1}
& 3 & 1.88e-02 & - & 2.51e-02 & - & 6.40e-02 & - \\ 
& 4 & 5.75e-03 & 1.71 & 7.13e-03 & 1.81 & 2.09e-02 & 1.61 \\ 
& 5 & 1.67e-03 & 1.78 & 2.02e-03 & 1.82 & 6.23e-03 & 1.75 \\ 
& 6 & 4.41e-04 & 1.92 & 5.33e-04 & 1.92 & 1.68e-03 & 1.89 \\ 
  \hline
\multirow{4}{4em}{$k$ = 1, $M$ = 2}
& 3 & 1.81e-02 & - & 2.53e-02 & - & 7.09e-02 & - \\ 
& 4 & 5.15e-03 & 1.81 & 6.98e-03 & 1.86 & 2.18e-02 & 1.70 \\ 
& 5 & 1.45e-03 & 1.82 & 1.92e-03 & 1.86 & 6.44e-03 & 1.76 \\ 
& 6 & 3.80e-04 & 1.94 & 4.99e-04 & 1.94 & 1.71e-03 & 1.91 \\ 
\hline
\multirow{4}{4em}{$k$ = 2, $M$ = 2}
& 3 & 2.13e-03 & - & 2.66e-03 & - & 1.01e-02 & - \\ 
& 4 & 5.11e-04 & 2.06 & 6.64e-04 & 2.00 & 2.77e-03 & 1.87 \\ 
& 5 & 1.08e-04 & 2.25 & 1.43e-04 & 2.21 & 5.42e-04 & 2.36 \\ 
& 6 & 1.71e-05 & 2.66 & 2.38e-05 & 2.59 & 9.55e-05 & 2.50 \\ 
\hline
\multirow{4}{4em}{$k$ = 2, $M$ = 3}
& 3 & 7.70e-04 & - & 1.03e-03 & - & 3.77e-03 & - \\ 
& 4 & 1.56e-04 & 2.31 & 2.05e-04 & 2.33 & 7.58e-04 & 2.31 \\ 
& 5 & 2.79e-05 & 2.48 & 3.63e-05 & 2.50 & 1.38e-04 & 2.45 \\ 
& 6 & 4.37e-06 & 2.67 & 5.99e-06 & 2.60 & 2.21e-05 & 2.65 \\ 
\hline
\multirow{4}{4em}{$k$ = 2, $M$ = 4}
& 3 & 7.98e-04 & - & 1.05e-03 & - & 3.61e-03 & - \\
& 4 & 1.62e-04 & 2.30 & 2.06e-04 & 2.34 & 7.53e-04 & 2.26 \\
& 5 & 2.83e-05 & 2.52 & 3.65e-05 & 2.50 & 1.38e-04 & 2.45 \\
& 6 & 4.40e-06 & 2.69 & 6.00e-06 & 2.60 & 2.21e-05 & 2.64 \\
\hline
\end{tabular}
\end{table}

Next, we test the convergence rate of the sparse grid DG method. We apply three different types of interpolation. The first one is the Lagrange interpolation with the interpolation points at the inner points of elements. The results are shown in Table \ref{ex:table-2D-Burgers-Sparse-Lagr-inner}, some {  instability} is observed for very fine mesh with $k=2.$ The second one is the the Lagrange interpolation with the interpolation points at the interface. The results are shown in Table \ref{ex:table-2D-Burgers-Sparse-Lagr-interface}. The results are better than the first type. For $k = 1$ and $M= 2$, the convergence order is around 1.5, as predicted. For $k = 2$, the convergence order is around 2 with $M= 2$. {  This indicates} that the condition in Proposition \ref{prop:accurate-interp} is necessary here. For $k = 2$ and $M= 3$, there still exists some {  instability} for very fine mesh. This motivates us to apply the Hermite interpolation in which we only use the end points. As shown in Table \ref{ex:table-2D-Burgers-Sparse-Herm}, the scheme with the Hermite interpolation is stable with predicted accuracy. This numerical experiment reveals that the Hermite interpolation is more stable than the Lagrange interpolation, and should be the method of choice.
\begin{table}[!hbp]
\centering
\caption{2D Burgers' equation at $T=0.01$, $\Delta t=0.1h$, sparse grid DG, Lagrange interpolation with inner interpolation points.}
\label{ex:table-2D-Burgers-Sparse-Lagr-inner}
\begin{tabular}{c|c|c|c|c|c|c|c}
  \hline
 poly degrees & $N$ & L$^1$-error & order & L$^2$-error & order & L$^{\infty}$-error & order \\
  \hline
\multirow{4}{4em}{$k$ = 1, $M$ = 1}
& 4 & 3.61e-02 & - & 4.90e-02 & - & 2.45e-01 & - \\ 
& 5 & 1.60e-02 & 1.17 & 2.18e-02 & 1.17 & 1.29e-01 & 0.93 \\ 
& 6 & 7.61e-03 & 1.07 & 1.02e-02 & 1.09 & 7.15e-02 & 0.85 \\ 
& 7 & 3.48e-03 & 1.13 & 4.71e-03 & 1.12 & 3.33e-02 & 1.10 \\ 
\hline
\multirow{4}{4em}{$k$ = 1, $M$ = 2}
& 4 & 2.89e-02 & - & 4.07e-02 & - & 2.47e-01 & - \\ 
& 5 & 8.70e-03 & 1.73 & 1.21e-02 & 1.75 & 8.07e-02 & 1.61 \\ 
& 6 & 2.98e-03 & 1.55 & 4.21e-03 & 1.52 & 3.03e-02 & 1.41 \\ 
& 7 & 1.13e-03 & 1.40 & 1.61e-03 & 1.38 & 1.10e-02 & 1.46 \\ 
\hline
\multirow{3}{4em}{$k$ = 2, $M$ = 2}
& 4 & 1.62e-02 & - & 2.61e-02 & - & 2.22e-01 & - \\ 
& 5 & 5.13e-02 & -1.67 & 1.13e-01 & -2.11 & 1.05e+00 & -2.23 \\ 
& 6 & 1.38e+06 & -24.68 & 4.53e+06 & -25.26 & 1.12e+08 & -26.67 \\ 
\hline
\multirow{4}{4em}{$k$ = 2, $M$ = 3}
& 4 & 3.37e-03 & - & 4.78e-03 & - & 4.77e-02 & - \\ 
& 5 & 7.12e-04 & 2.24 & 1.03e-03 & 2.21 & 9.49e-03 & 2.33 \\ 
& 6 & 3.71e-04 & 0.94 & 6.44e-04 & 0.68 & 7.64e-03 & 0.31 \\ 
& 7 & 2.13e-03 & -2.52 & 4.31e-03 & -2.74 & 6.56e-02 & -3.10 \\ 
\hline
\end{tabular}
\end{table}

\begin{table}[!hbp]
\centering
\caption{2D Burgers' equation at $T=0.01$, $\Delta t=0.1h$, sparse grid DG, Lagrange interpolation with interface interpolation points.}
\label{ex:table-2D-Burgers-Sparse-Lagr-interface}
\begin{tabular}{c|c|c|c|c|c|c|c}
  \hline
 poly degrees & $N$ & L$^1$-error & order & L$^2$-error & order & L$^{\infty}$-error & order \\
  \hline
\multirow{4}{4em}{$k$ = 1, $M$ = 1}
& 4 & 3.58e-02 & - & 4.87e-02 & - & 2.45e-01 & - \\ 
& 5 & 1.58e-02 & 1.18 & 2.15e-02 & 1.18 & 1.27e-01 & 0.95 \\ 
& 6 & 7.47e-03 & 1.08 & 1.00e-02 & 1.10 & 7.05e-02 & 0.85 \\ 
& 7 & 3.40e-03 & 1.14 & 4.60e-03 & 1.13 & 3.30e-02 & 1.09 \\ 
\hline
\multirow{4}{4em}{$k$ = 1, $M$ = 2}
& 4 & 2.84e-02 & - & 4.01e-02 & - & 2.38e-01 & - \\ 
& 5 & 8.14e-03 & 1.80 & 1.13e-02 & 1.83 & 7.49e-02 & 1.67 \\ 
& 6 & 2.70e-03 & 1.59 & 3.88e-03 & 1.54 & 2.73e-02 & 1.46 \\ 
& 7 & 8.07e-04 & 1.74 & 1.17e-03 & 1.73 & 1.10e-02 & 1.31 \\ 
\hline
\multirow{4}{4em}{$k$ = 2, $M$ = 2}
& 4 & 5.74e-03 & - & 7.71e-03 & - & 2.79e-02 & - \\ 
& 5 & 2.68e-03 & 1.10 & 3.35e-03 & 1.20 & 1.05e-02 & 1.41 \\ 
& 6 & 6.14e-04 & 2.13 & 7.70e-04 & 2.12 & 3.51e-03 & 1.58 \\ 
& 7 & 1.59e-04 & 1.95 & 2.03e-04 & 1.92 & 1.21e-03 & 1.54 \\ 
\hline
\multirow{4}{4em}{$k$ = 2, $M$ = 3}
& 4 & 3.37e-03 & - & 4.78e-03 & - & 4.72e-02 & - \\ 
& 5 & 7.09e-04 & 2.25 & 1.02e-03 & 2.22 & 9.25e-03 & 2.35 \\ 
& 6 & 3.49e-04 & 1.02 & 6.03e-04 & 0.77 & 7.12e-03 & 0.38 \\ 
& 7 & 1.70e-03 & -2.29 & 3.44e-03 & -2.51 & 5.21e-02 & -2.87 \\ 
\hline
\end{tabular}
\end{table}

\begin{table}[!hbp]
\centering
\caption{{  2D Burgers' equation at $T=0.01$, $\Delta t=0.1h$ for $k=1,2$ and $\Delta t=0.1h^{4/3}$ for $k=3$, sparse grid DG, Hermite interpolation.}}
\label{ex:table-2D-Burgers-Sparse-Herm}
\begin{tabular}{c|c|c|c|c|c|c|c}
  \hline
 poly degrees & $N$ & L$^1$-error & order & L$^2$-error & order & L$^{\infty}$-error & order \\
  \hline
\multirow{4}{4em}{$k$ = 1, $M$ = 3}
& 5 & 8.30e-03 & - & 1.15e-02 & - & 7.36e-02 & - \\ 
& 6 & 2.85e-03 & 1.54 & 3.99e-03 & 1.52 & 2.59e-02 & 1.51 \\ 
& 7 & 8.81e-04 & 1.70 & 1.23e-03 & 1.70 & 1.03e-02 & 1.33 \\ 
& {  8} & {  2.66e-04} & {  1.73} & {  3.72e-04} & {  1.73} & {  3.69e-03} & {  1.48} \\ 
\hline
\multirow{4}{4em}{$k$ = 2, $M$ = 3}
& 5 & 1.10e-03 & - & 1.47e-03 & - & 5.30e-03 & - \\ 
& 6 & 1.60e-04 & 2.78 & 2.11e-04 & 2.80 & 9.04e-04 & 2.55 \\ 
& 7 & 2.75e-05 & 2.54 & 3.67e-05 & 2.52 & 2.23e-04 & 2.02 \\ 
& {  8} & {  6.25e-06} & {  2.14} & {  8.41e-06} & {  2.13} & {  4.98e-05} & {  2.16} \\ 
\hline
\multirow{4}{4em}{$k$ = 2, $M$ = 5}
& 5 & 3.85e-04 & - & 5.45e-04 & - & 3.07e-03 & - \\ 
& 6 & 8.18e-05 & 2.23 & 1.25e-04 & 2.12 & 1.12e-03 & 1.45 \\ 
& 7 & 1.37e-05 & 2.58 & 1.99e-05 & 2.65 & 2.26e-04 & 2.31 \\ 
& {  8} & {  2.30e-06} & {  2.57} & {  3.41e-06} & {  2.55} & {  2.90e-05} & {  2.97} \\
\hline
\multirow{4}{4em}{$k$ = 3, $M$ = 5}
& 5 & 4.37e-05 & - & 6.85e-05 & - & 2.86e-04 & - \\ 
& 6 & 3.75e-06 & 3.54 & 6.19e-06 & 3.47 & 6.26e-05 & 2.19 \\ 
& 7 & 2.81e-07 & 3.74 & 4.46e-07 & 3.80 & 3.16e-06 & 4.31 \\
& {  8} & {  2.74e-08} & {  3.35} & {  4.44e-08} & {  3.33} & {  6.30e-07} & {  2.33} \\
\hline
\end{tabular}
\end{table}

Next, we discuss the convergence rate with  adaptivity. Following \cite{guo2017adaptive}, two types rates of convergence are calculated. The first one is the convergence rate with respect to the error thresold:
\begin{equation*}
    R_{\epsilon_l} = \frac{\log(e_{l-1}/e_l)}{\log(\epsilon_{l-1}/\epsilon_l)}.
\end{equation*}
The second one is the convergence rate with respect to degrees of freedom:
\begin{equation*}
    R_{\textrm{DoF}_l} = \frac{\log(e_{l-1}/e_l)}{\log(\textrm{DoF}_{l-1}/\textrm{DoF}_l)}.
\end{equation*}
We run the simulations with a fixed maximum mesh level $N=8$ and different $\epsilon$ values, and we report the $L^2$ errors and the number of active degrees of freedom at $T=0.01$ in Table \ref{ex:table-2D-Burgers-Rate}. We observe similar convergence rates as in Table 1 in \cite{guo2017adaptive}: $R_{\epsilon}$ is slightly smaller than 1, and $R_{\textrm{DoF}}$ is much larger than $(k+1)/2$ ($R_{\textrm{DoF}}$ for the standard adaptive DG scheme for 2D problems) but still smaller than $(k+1)$. This demonstrates the effectiveness of the multiresolution adaptive algorithm. Sparsity is indeed achieved for smooth solutions. 
\begin{table}[!hbp]
\centering
\caption{2D Burgers' equation at $T=0.01$. Convergence rates with respect to the error {  threshold} and degrees of freedom.}
\label{ex:table-2D-Burgers-Rate}
\begin{tabular}{c|c|c|c|c|c}
  \hline
& $\epsilon$ & DoF & L$^2$-error & $R_{\textrm{DoF}}$ & $R_{\epsilon}$ \\
  \hline
\multirow{4}{4em}{$k = 1$}
& 1e-04 & 3488 & 3.89e-04 & - & - \\ 
& 5e-05 & 4608 & 2.95e-04 & 0.98 & 0.40 \\ 
& 1e-05 & 9408 & 9.29e-05 & 1.62 & 0.72 \\ 
& 5e-06 & 12272 & 4.79e-05 & 2.49 & 0.95 \\ 
\hline
\multirow{4}{4em}{$k = 2$}
& 1e-04 & 1656 & 2.02e-04 & - & - \\ 
& 5e-05 & 1908 & 1.20e-04 & 3.71 & 0.76 \\ 
& 1e-05 & 3600 & 3.83e-05 & 1.79 & 0.71 \\ 
& 5e-06 & 4068 & 2.33e-05 & 4.09 & 0.72 \\ 
\hline
\multirow{4}{4em}{$k = 3$}
& 1e-04 & 1152 & 5.52e-05 & - & - \\ 
& 5e-05 & 1472 & 2.49e-05 & 3.25 & 1.15 \\ 
& 1e-05 & 1920 & 1.16e-05 & 2.89 & 0.48 \\ 
& 5e-06 & 2624 & 5.62e-06 & 2.31 & 1.04 \\ 
\hline
\end{tabular}
\end{table}

Next, we test the ability of our scheme on capturing non-smooth solutions. We also take the same initial value. The shock begins to develop at $t=\frac{1}{4\pi}\approx0.07958$. The numerical results at $t=0.2$ are shown in Fig. \ref{fig:burgers2d-shock}. We observe the numerical solution coincides with the exact solution very well. The elements with non-zero artificial viscosity are also concentrated near the shock.
\begin{figure}
    \centering
    \subfigure[numerical solution (surface plot)]{
    \begin{minipage}[b]{0.46\textwidth}
    \includegraphics[width=1\textwidth]{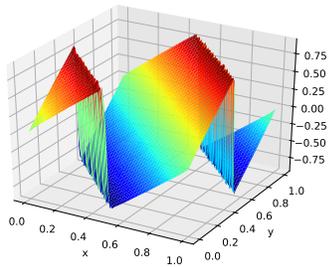}
    \end{minipage}
    }
    \subfigure[numerical solution (contour plot)]{
    \begin{minipage}[b]{0.46\textwidth}    
    \includegraphics[width=1\textwidth]{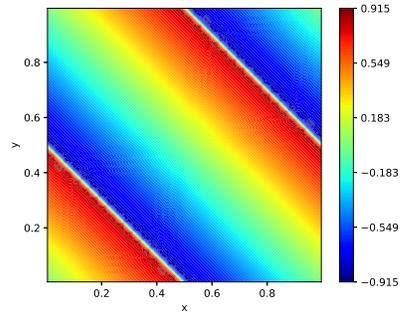}
    \end{minipage}
    }
    \bigskip
    \subfigure[numerical solution in 1D cut along diagnal]{
    \begin{minipage}[b]{0.46\textwidth}
    \includegraphics[width=1\textwidth]{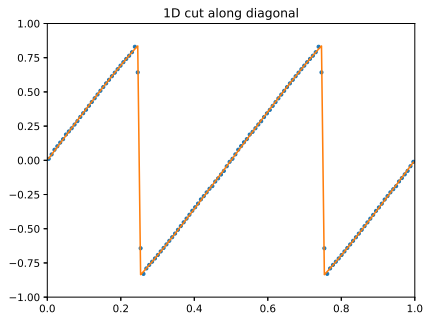}
    \end{minipage}
    }
    \subfigure[elements with   artificial viscosity]{
    \begin{minipage}[b]{0.46\textwidth}    
    \includegraphics[width=1\textwidth]{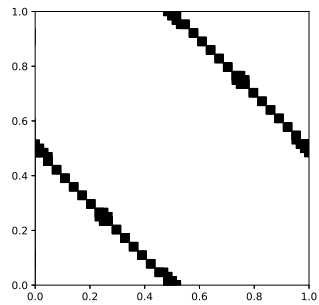}
    \end{minipage}
    }    
    \caption{Example \ref{exam:burgers-2d}: 2D Burgers' equation at $t=0.2$. $N=7$ and $\epsilon=5e-4$. (a) numerical solution (surface plot); (b) numerical solution (contour plot); (c) numerical solution in 1D cut along diagnal; (d) elements with   artificial viscosity.}
    \label{fig:burgers2d-shock}
\end{figure}

{  Next, we compare the efficiency of our adaptive method with the non-adaptive (full grid) DG scheme for non-smooth solutions in Fig. \ref{fig:burgers2d-full}. Here,   the full grid scheme with $N=6$   has 36864 DoF. The adaptive DG method has on average 19086 DoF and at most 22000 DoF, which is fewer than the full grid method. However, the adaptive is more accurate than the full grid as can be seen in Fig. \ref{fig:burgers2d-full}.
\begin{figure}
    \centering
    \subfigure[profile in 1D cut along diagnal]{
    \begin{minipage}[b]{0.46\textwidth}
    \includegraphics[width=1\textwidth]{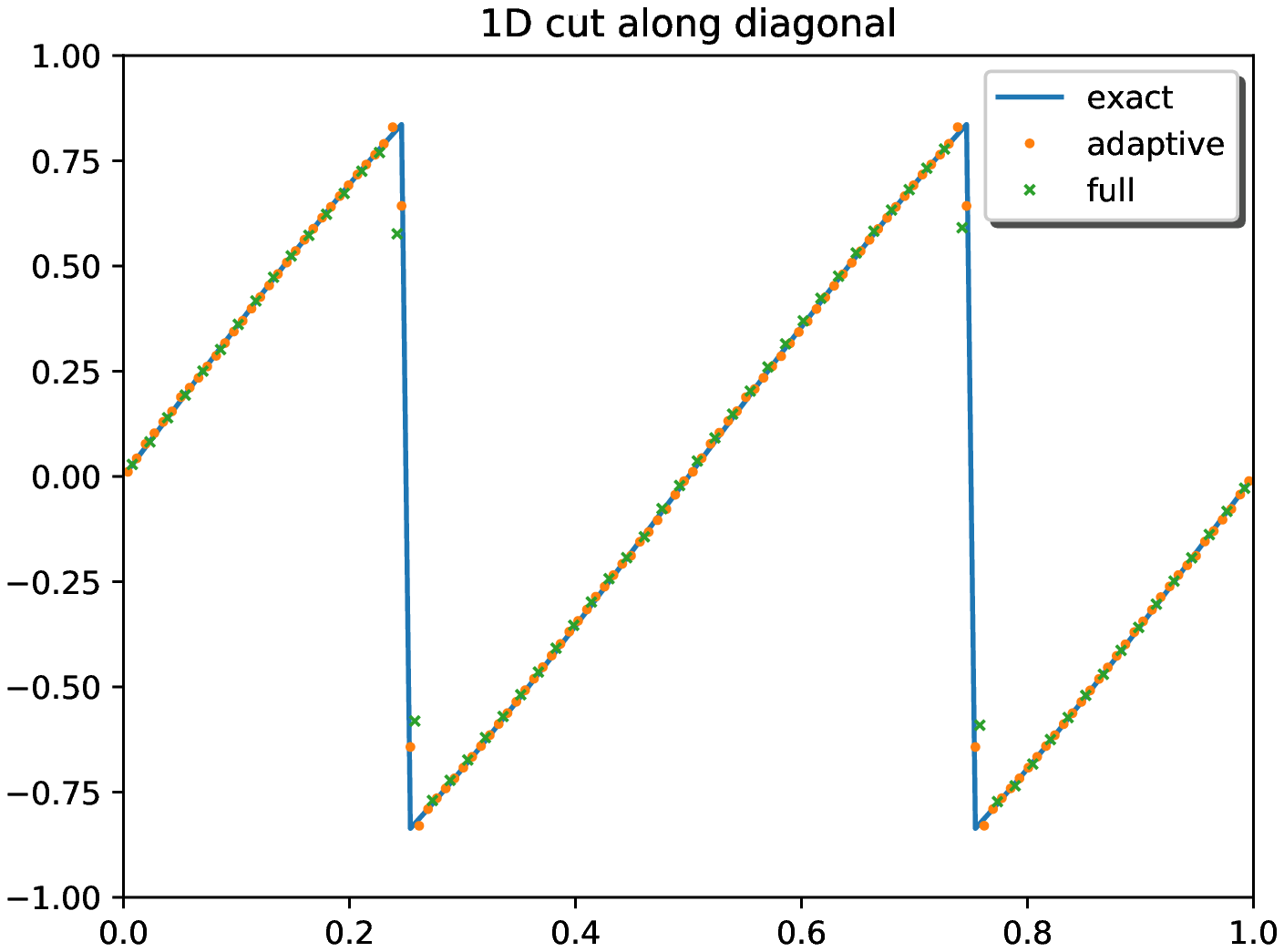}    
    \end{minipage}
    }
    \subfigure[error between numerical solution and exact solution]{
    \begin{minipage}[b]{0.46\textwidth}    
    \includegraphics[width=1\textwidth]{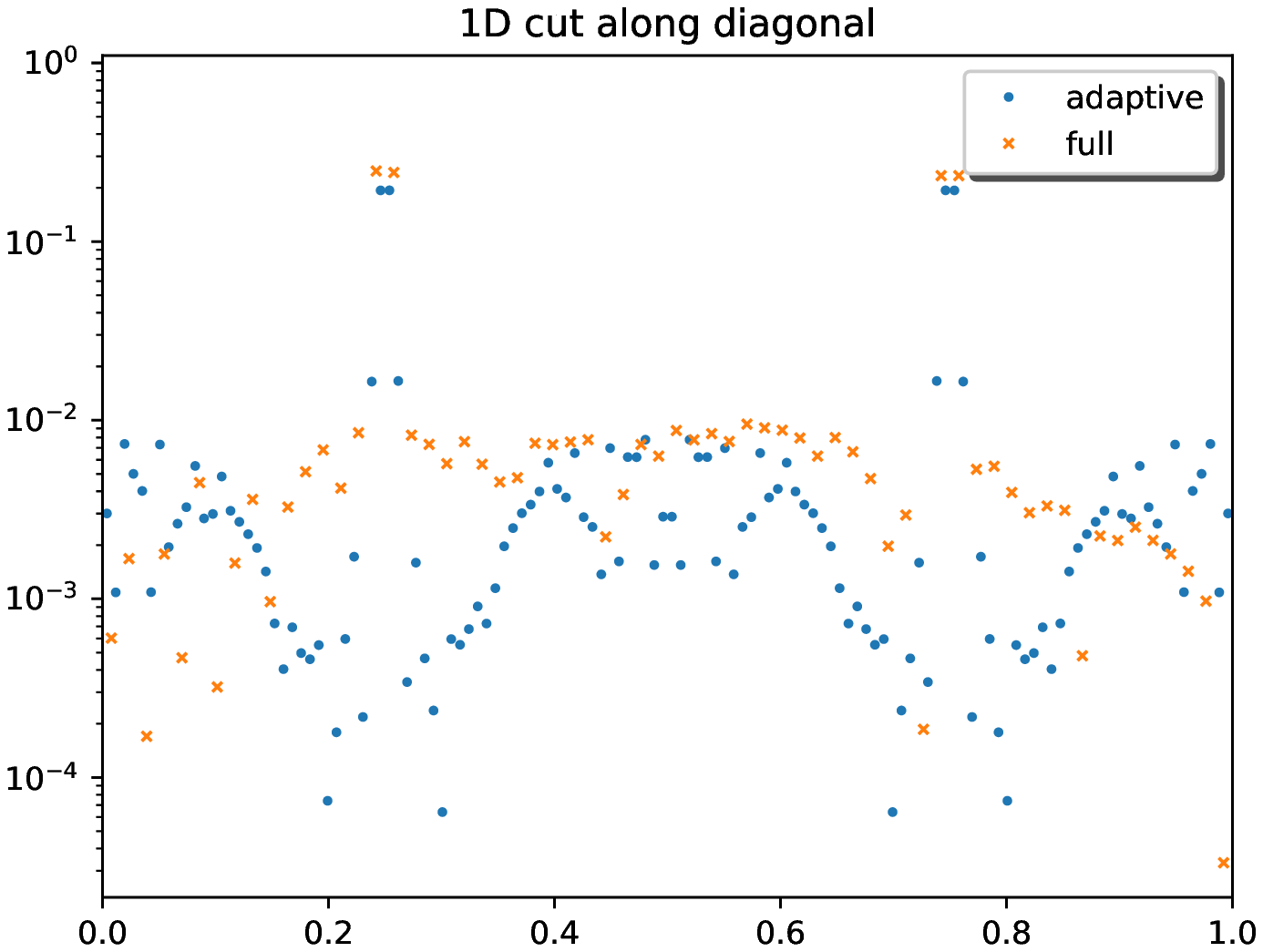}
    \end{minipage}
    }
    \caption{Example \ref{exam:burgers-2d}: 2D Burgers' equation at $t=0.2$. Full grid (non-adaptive) DG with $N=6$ vs. adaptive DG with $\epsilon=5e-4$. Left: profile in 1D cut along diagnal; right: error between numerical solution and exact solution.}
    \label{fig:burgers2d-full}
\end{figure}

We also investigate the efficiency of   adaptive method for smooth solutions in Fig. \ref{fig:burgers2d-dof-compare}. To reach the same accuracy, the adaptive method has much less DoF than the non-adaptive method. This is because of the sparse grid nature of our method.
\begin{figure}
    \centering
    \includegraphics[width=0.5\textwidth]{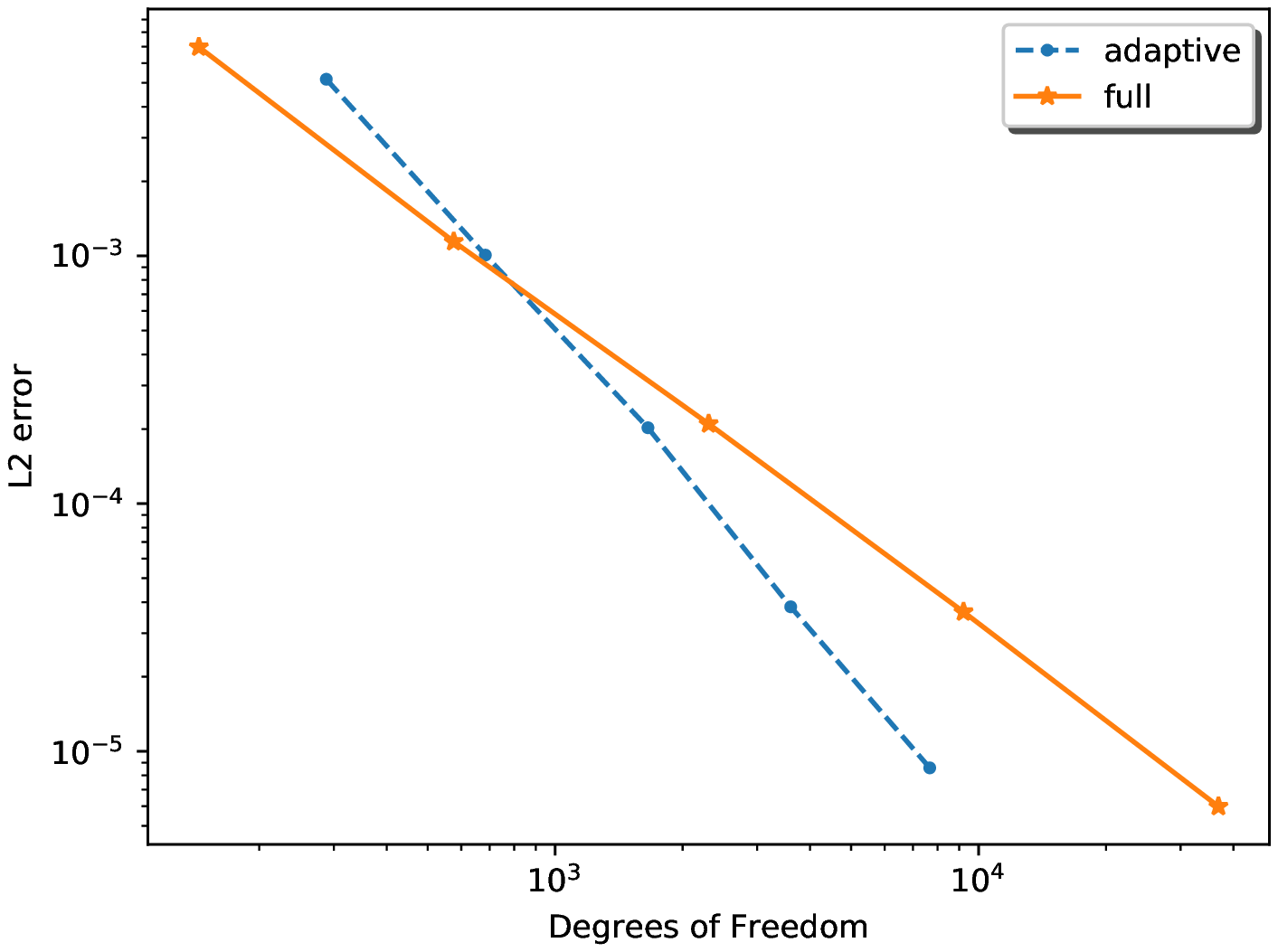}
    \caption{Example \ref{exam:burgers-2d}: 2D Burgers' equation at $t=0.01$. Full grid (non-adaptive) DG with $N=2,3,4,5,6$ vs. adaptive DG with $\epsilon=10^{-2},10^{-3},\dots,10^{-6}$ and fixed $N=8$.}
    \label{fig:burgers2d-dof-compare}
\end{figure}
}
\end{examp}



\begin{examp}[2D KPP rotating wave problem]\label{exam:kpp}
In the last example, we consider the 2D KPP rotating wave problem with the non-convex physical flux:
\begin{equation*}
    u_t + \sin(u)_x + \cos(u)_y = 0.
\end{equation*}
The initial condition is
\begin{equation*}
    \begin{split}
    u_0(x,y) = \left\{
    \begin{aligned}
    & 3.5\pi , \quad & (x-\half)^2+(y-\half)^2\le \frac{1}{16}, \\ 
    & 0.25\pi, \quad & \textrm{otherwise}.
    \end{aligned}
    \right.
    \end{split}
\end{equation*}
This is a rather challenging test case proposed in \cite{kurganov2007adaptive}, since the flux is non convex and a two-dimensional composite wave structure is present.
The code is run up to $t=0.4$. The maximum mesh level is $N=7$ and the error thresold is $\epsilon=5\times10^{-4}$. The numerical solutions and elements with non-zero artificial viscosity at $t=0.2$ and $t=0.4$ are shown in Fig. \ref{fig:kpp}. Our numerical scheme can capture the wave structure very well.

\begin{figure}
    \centering
    \subfigure[numerical solution in 2D at $t=0.2$]{
    \begin{minipage}[b]{0.46\textwidth}
    \includegraphics[width=1\textwidth]{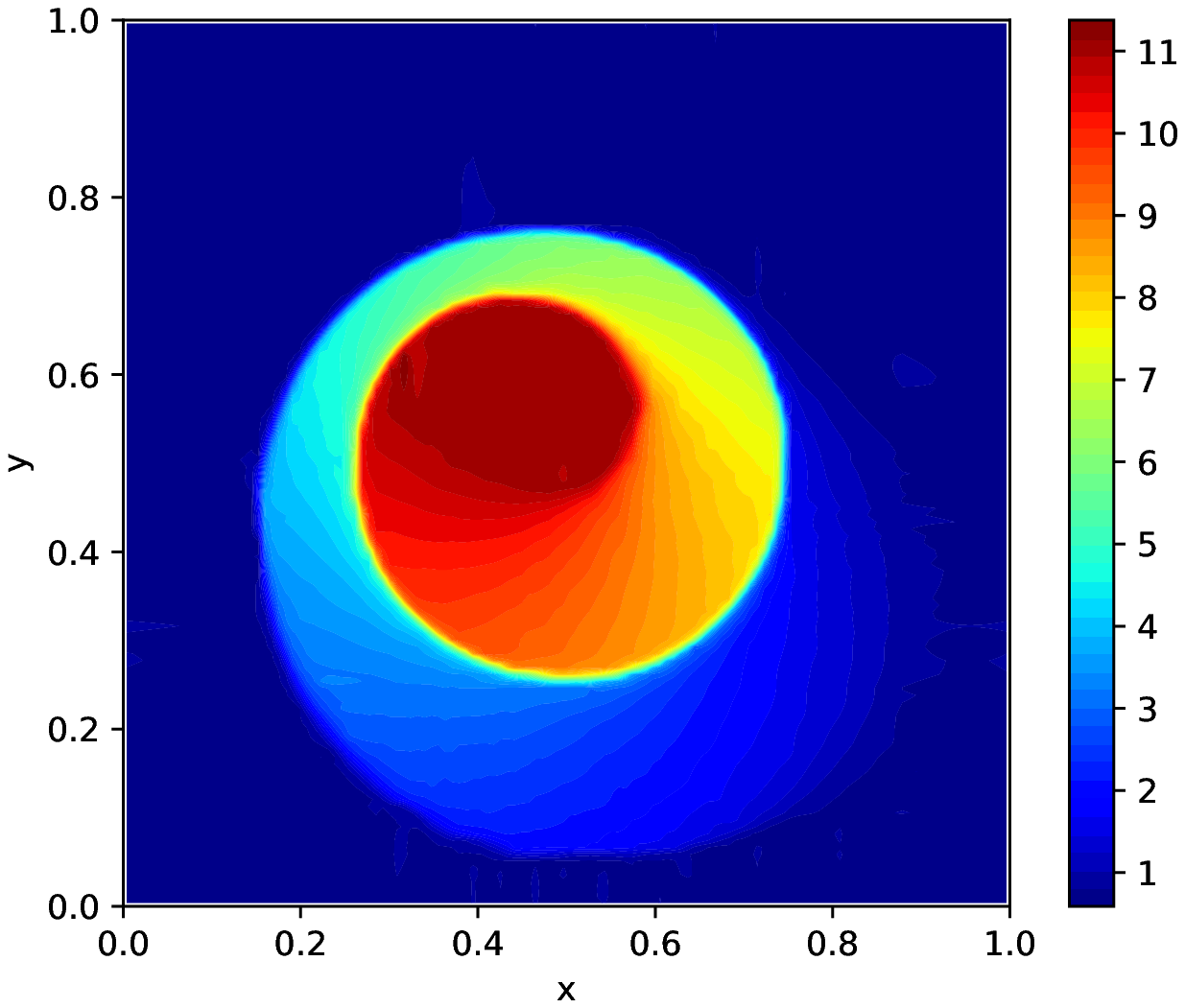}
    \end{minipage}
    }
    \subfigure[elements with artificial viscosity at $t=0.2$]{
    \begin{minipage}[b]{0.46\textwidth}    
    \includegraphics[width=1\textwidth]{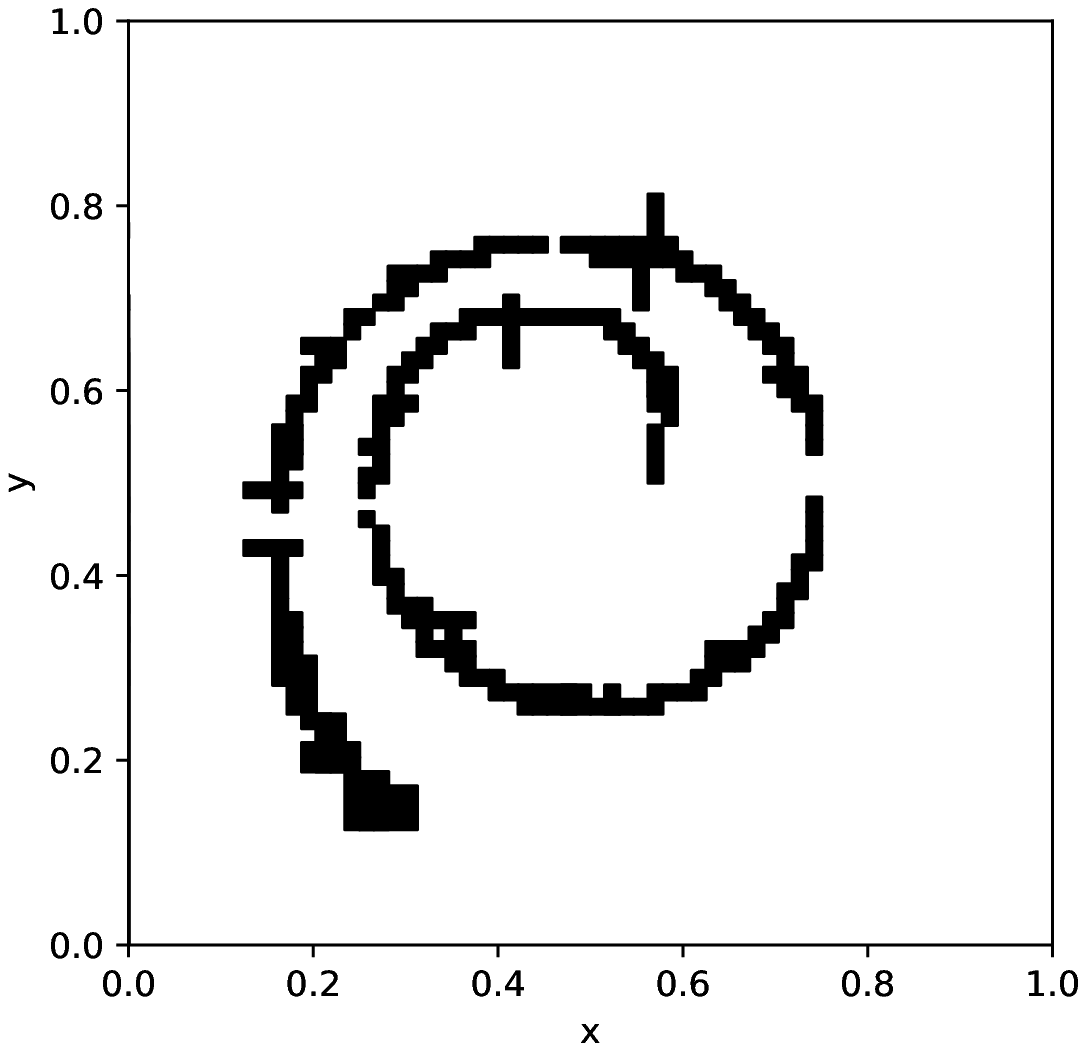}
    \end{minipage}
    }
    \bigskip
    \subfigure[numerical solution in 2D at $t=0.4$]{
    \begin{minipage}[b]{0.46\textwidth}
    \includegraphics[width=1\textwidth]{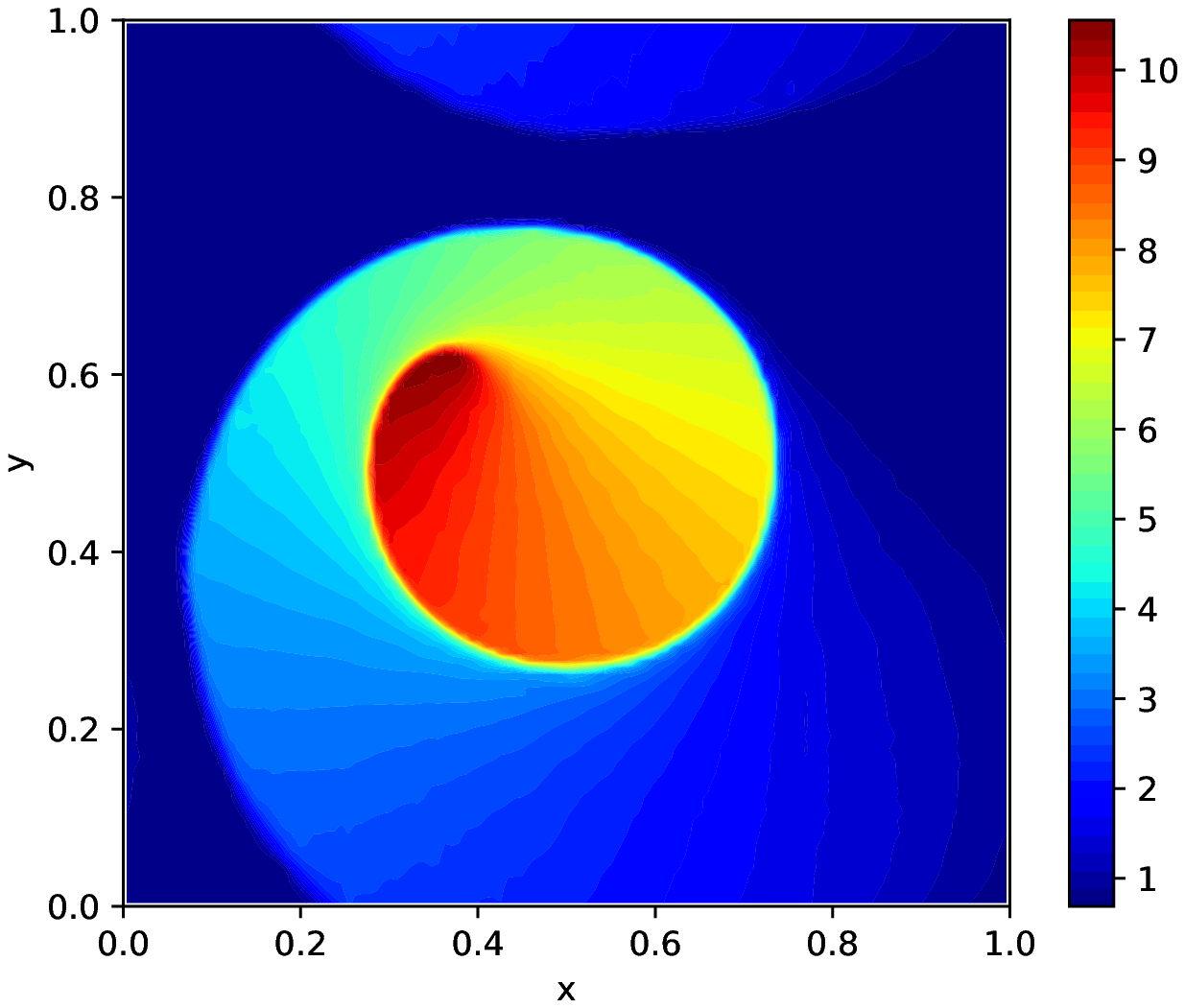}
    \end{minipage}
    }
    \subfigure[elements with artificial viscosity at $t=0.4$]{
    \begin{minipage}[b]{0.46\textwidth}    
    \includegraphics[width=1\textwidth]{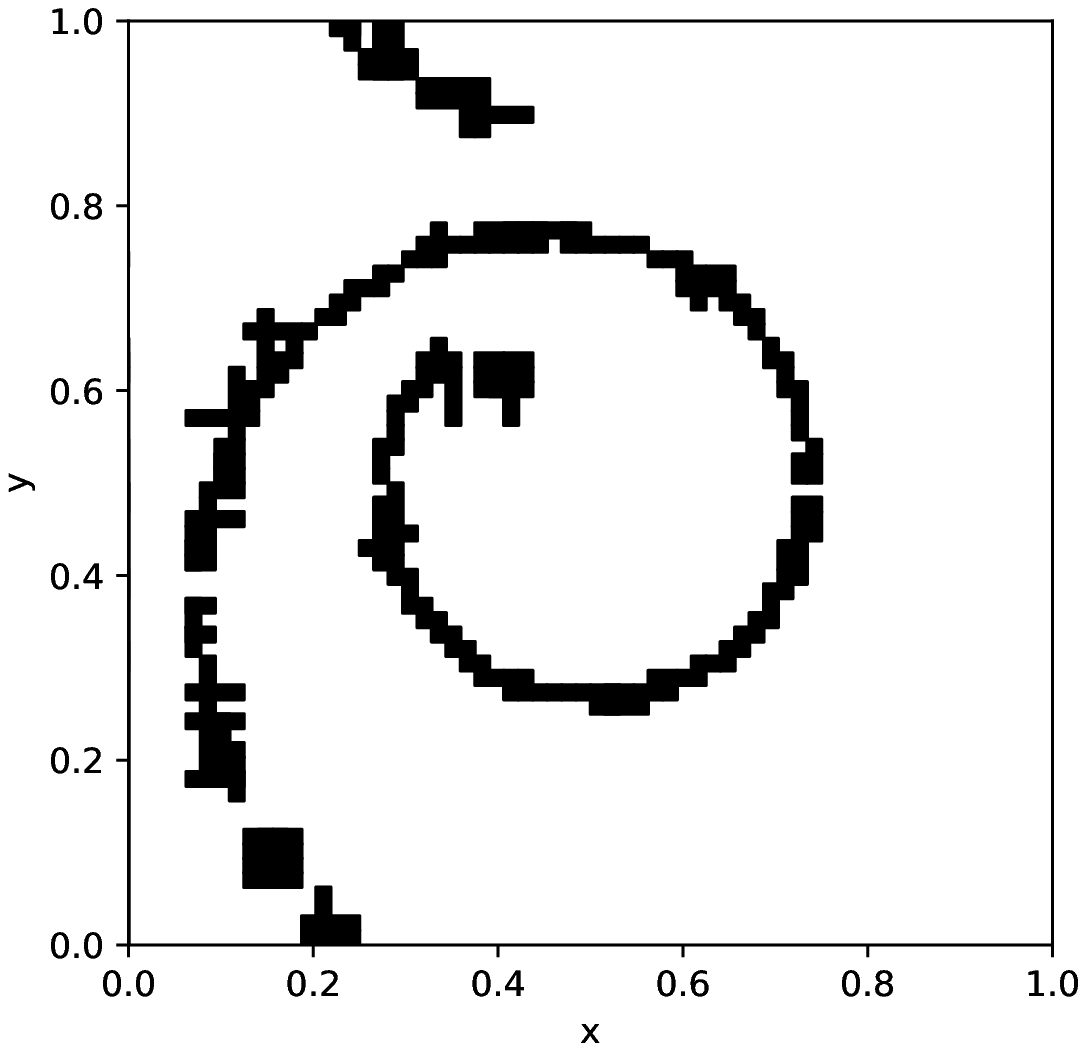}
    \end{minipage}
    }
    \caption{Example \ref{exam:kpp}: 2D KPP rotating wave problem at $t=0.2$ and $t=0.4$. $N=7$ and $\epsilon=5e-4$. Left: numerical solution in 2D; right: elements with   artificial viscosity.}
    \label{fig:kpp}
\end{figure}

\end{examp}



\section{Concluding remarks}

In this paper, we propose an adaptive multiresolution  DG scheme for scalar hyperbolic conservation laws in multidimensions. Besides the Alpert's multiwavelets, the interpolatory multiwavelets are applied to treat the nonlinear integrals over elements and edges in DG schemes. From numerical study, we find that the multiresolution Hermite interpolation is the most stable. Because of the coherence of the multiresolution interpolation with the MRA of the numerical solution, our method    can achieve similar computational complexity as the sparse grid DG method for smooth solutions \cite{guo2016transport, guo2017adaptive}. Artificial viscosity and adaptivity are activated for non-smooth solutions. The required DoF  corresponds to the intrinsic complexity of the solution structure, and   artificial viscosity is only added at locations near the shock maintaining   sharpness of the solution profile.

{  The motivation of this work is for efficient computations of nonlinear PDE problems in high dimensions. The construction of the numerical schemes in this paper will be  extended to Hamilton-Jacobi-Bellman equations and kinetic equations in the future. The adaptive mechanism can also be incorporated for multiscale kinetic-fluid simulations.
}

\section*{Acknowledgements}

We would like to thank Chi-Wang Shu from Brown University for providing the CFL constants of local DG method for diffusion equations and many fruitful discussions. We also would like to thank Kai Huang from Michigan State University, Wei Guo from Texas Tech University, Yuan Liu from Wichita State University, Zhanjing Tao from Jilin University in China, and Qi Tang from Los Alamos National Laboratory for the assistance and discussion in code implementation.

\appendix

\section{Lagrange interpolation basis functions}

For the completeness of our paper, we present the multiresolution  interpolation basis functions, which are first introduced in \cite{tao2019collocation}. In this part, we focus on the Lagrange interpolantion, i.e. $K=0.$  The case in which the interpolation points {  are imposed} in the inner domain, as implemented in Table \ref{ex:table-2D-Burgers-Sparse-Lagr-inner} is discussed first, followed by   the case in which the points are imposed at the cell interface, see the corresponding numerical results in Table \ref{ex:table-2D-Burgers-Sparse-Lagr-interface}.

The basis functions in $\tilde{W}_1$ are piecewise polynomials on $I_l:=(0,\half)$ and $I_r:=(\half,1)$. Note that the functions may be discontinuous at the interface $x=1/2$, thus $I_l$ and $I_r$ are  both  defined to be open intervals. The basis functions in $\tilde{W}_1$ in this paper are all supported on one half interval $I_l$ or $I_r$ and vanish on the other half. For simplicity, we will only declare the function on its support. For example, $\psi_0(x)|_{I_r}$ gives the definition of $\psi_0$ on $I_r$ and indicates that $\psi_0$ vanishes on $I_l$.

\subsection{interpolation points in the inner domain}\label{subsec:append-lagrange-inner}

\subsubsection{$P=1$ and $K=0$}
The interpolation points are
\begin{equation*}
    \tilde{X}_0 = \{ \frac{1}{3}, \frac{2}{3} \}, \quad \tilde{X}_1 = \{ \frac{1}{6}, \frac{5}{6} \}.
\end{equation*}
The basis functions in $\tilde{W}_0^1$ and $\tilde{W}_1^1$ are
\begin{align*} 
\begin{array}{ll}
\phi_{0}(x)=-3x+2,  & \phi_{1}(x)=3x-1.\\
\psi_{0}(x)|_{I_l}= -6x+2, & \psi_{1}(x)|_{I_r}= 6x-4. \\
\end{array} 
\end{align*}

\subsubsection{$P=2$ and $K=0$}
The interpolation points are
\begin{equation*}
    \tilde{X}_0 =  \{ \frac{1}{6}, \frac{1}{3}, \frac{2}{3} \}, \quad \tilde{X}_1 =  \{ \frac{1}{12}, \frac{7}{12}, \frac{5}{6} \}.
\end{equation*}
The basis functions in $\tilde{W}_0^2$ and $\tilde{W}_1^2$ are
\begin{align*} 
\begin{array}{ll}
\phi_{0}(x)=\frac{4}{3}(3x-2)(3x-1),  \quad  \phi_{1}(x)=-(3x-2)(6x-1), \\ 
\phi_2(x)=\frac{1}{3}(3x-1)(6x-1).		
\end{array} 
\end{align*}
and
\begin{align*} 
\begin{array}{ll}
\psi_{0}(x)|_{I_l} = \frac{8}{3}(3x-1)(6x-1), &
\psi_{1}(x)|_{I_r} = \frac{8}{3}(3x-2)(6x-5), \\
\psi_{2}(x)|_{I_r} = \frac{2}{3}(3x-2)(12x-7).
\end{array} 
\end{align*}

\subsubsection{$P=3$ and $K=0$}

The interpolation points are
\begin{equation*}
    \tilde{X}_0 =  \{ \frac{1}{5}, \frac{2}{5}, \frac{3}{5}, \frac{4}{5} \}, \quad \tilde{X}_1 =  \{ \frac{1}{10}, \frac{3}{10}, \frac{7}{10}, \frac{9}{10} \}.
\end{equation*}
The basis functions in $\tilde{W}_0^3$ and $\tilde{W}_1^3$ are
\begin{align*} 
\begin{array}{ll}
\phi_0(x) = -\frac{1}{6}(5x-4)(5x-3)(5x-2), & \phi_1(x) = \frac{1}{2}(5x-4)(5x-3)(5x-1), \\
\phi_2(x) = -\frac{1}{2}(5x-4)(5x-2)(5x-1), & \phi_3(x) = \frac{1}{6}(5x-3)(5x-2)(5x-1),
\end{array} 
\end{align*}
and
\begin{align*} 
\begin{array}{l}
\psi_{0}(x)|_{I_l}= -\frac{2}{3}(5x-2)(5x-1)(10x-3), \\
\psi_{1}(x)|_{I_l}= -2(5x-2)(5x-1)(10x-1), \\ 
\psi_{2}(x)|_{I_r}= 2(5x-4)(5x-3)(10x-9), \\
\psi_{3}(x)|_{I_r}= \frac{2}{3}(5x-4)(5x-3)(10x-7).	
\end{array}
\end{align*}

\subsection{interpolation points at the interface}\label{subsec:append-lagrange-interface}

\subsubsection{$P=1$ and $K=0$}

The interpolation points are
\begin{equation*}
    \tilde{X}_0 =  \{ 0^+, 1^- \}, \quad \tilde{X}_1 =  \{ (\frac{1}{2})^-, (\frac{1}{2})^+ \}.
\end{equation*}
Here and below, we use superscripts $+,-$ to emphasize the left and right limits of a function at that point. This is a feature of the    discontinuous piecewise polynomial space.

The basis functions in $\tilde{W}_0^1$ and $\tilde{W}_1^1$ are
\begin{align*} 
\begin{array}{ll}
\phi_{1}(x)=-x+1,  & \phi_{2}(x)=x, \\
\psi_{1}(x)|_{I_l}= 2x, & \psi_{2}(x)|_{I_r} = -2x+2.
\end{array}
\end{align*}

\subsubsection{$P=2$ and $K=0$}
The interpolation points are
\begin{equation*}
    \tilde{X}_0 =  \{ 0^+, (\frac{1}{2})^-, 1^- \}, \quad \tilde{X}_1 =  \{ (\frac{1}{4})^-, (\frac{1}{2})^+, (\frac{3}{4})^- \}.
\end{equation*}
The basis functions in $\tilde{W}_0^2$ and $\tilde{W}_1^2$ are
\begin{align*} 
\begin{array}{lll}
\phi_{1}(x)=2(x-\frac{1}{2})(x-1),  &  \phi_{2}(x)= -4x(x-1), & \phi_{3}(x)=2x(x-\frac{1}{2}),
\end{array} 
\end{align*}
and
\begin{align*} 
\begin{array}{ll}
\psi_{0}(x)|_{I_l} = -16x(x-\frac{1}{2}),  &
\psi_{1}(x)|_{I_r} = 8(x-\frac{3}{4})(x-1),  \\
\psi_{2}(x)|_{I_r} = -16(x-\frac{1}{2})(x-1).
\end{array} 
\end{align*}

\subsubsection{$P=3$ and $K=0$}

The interpolation points are
\begin{equation*}
    \tilde{X}_0 =  \{ 0^+, (\frac{1}{4})^-, (\frac{1}{2})^-, 1^- \}, \quad \tilde{X}_1 =  \{ (\frac{1}{8})^-, (\frac{1}{2})^+, (\frac{5}{8})^-, (\frac{3}{4})^- \}.
\end{equation*}
The basis functions in $\tilde{W}_0^3$ and $\tilde{W}_1^3$ are
\begin{align*} 
\begin{array}{ll}
    \phi_0(x)=-(x-1)(2x-1)(4x-1), & \phi_1(x)=\frac{32}{3}x(x-1)(2x-1), \\
    \phi_2(x)=-4x(x-1)(4x-1), & \phi_3(x)=\frac{1}{3}x(2x-1)(4x-1),
\end{array} 
\end{align*}
and
\begin{align*} 
\begin{array}{ll}
\psi_{0}(x)|_{I_l} = \frac{64}{3} x (2 x-1) (4x-1),  &
\psi_{1}(x)|_{I_r} = -2 (x-1) (4x-3) (8x-5), \\
\psi_{2}(x)|_{I_r} = \frac{64}{3} (x-1) (2x-1) (4x-3), &
\psi_{3}(x)|_{I_r} = -8 (x-1 ) (2x-1) (8x-5).
\end{array} 
\end{align*}

\section{Hermite interpolation basis functions}\label{sec:append-hermite}
The Hermite interpolation basis functions are presented here. 
The interpolation points are put at the cell interface:
\begin{equation*}
    \tilde{X}_0 = \{ 0^+, 1^- \}, \quad \tilde{X}_1 = \{ (\frac{1}{2})^-, (\frac{1}{2})^+ \}.
\end{equation*}

\subsection{$P=1$ and $K=1$}

The basis functions in $\tilde{W}_0^3$ and $\tilde{W}_1^3$ are
\begin{align*} 
\begin{array}{ll}
	\phi_{0,0}(x) = (x-1)^2(2x+1), & \phi_{1,0}(x) = -x^2(2x-3), \\ 
	\phi_{0,1}(x) = x(x-1)^2, & \phi_{1,1}(x) = x^2(x-1).
\end{array} 
\end{align*}
and
\begin{align*} 
\begin{array}{ll}
\psi_{0,0}(x)|_{I_l} = -4x^2(4x-3), &
\psi_{1,0}(x)|_{I_r} = 4(x-1)^2(4x-1), \\
\psi_{0,1}(x)|_{I_l} = 2x^2(2x-1), &
\psi_{1,1}(x)|_{I_r} = 2(x-1)^2(2x-1),
\end{array} 
\end{align*}

\subsection{$P=1$ and $K=2$}
The basis functions in $\tilde{W}_0^5$ and $\tilde{W}_1^5$ are
\begin{align*} 
\begin{array}{ll}
	\phi_{0,0}(x) = -(x-1)^3(6x^2+3x+1), & \phi_{0,1}(x) = -x(x-1)^3(3x+1),\\
	\phi_{0,2}(x) = -\frac{1}{2}x^2(x-1)^3, &  \phi_{1,0}(x) = x^3(6x^2-15x+10),  \\
	\phi_{1,1}(x) = -x^3(x-1)(3x-4), & \phi_{1,2}(x) = \frac{1}{2}x^3(x-1)^2.
\end{array} 
\end{align*}
and
\begin{align*} 
\begin{array}{ll}
\psi_{0,0}(x)|_{I_l} = 16 x^3 (12 x^2 - 15 x +5), &
\psi_{1,0}(x)|_{I_r} = -16 (x-1)^3 (12 x^2 - 9 x + 2), \\
\psi_{0,1}(x)|_{I_l} = -8 x^3 ( 2 x-1 ) (3 x-2 ), &
\psi_{1,1}(x)|_{I_r} = -8 (x-1)^3 (2 x-1) (3 x-1), \\
\psi_{0,2}(x)|_{I_l} = x^3(2x-1)^2, &
\psi_{1,2}(x)|_{I_r} = - (x-1)^3 (2 x-1)^2.					
\end{array} 
\end{align*}


\newpage

\bibliographystyle{abbrv}

\end{document}